\documentclass[11pt]{amsart}
\usepackage{amssymb,latexsym,amsmath,longtable,mathdots}
\usepackage[mathscr]{eucal}
\usepackage{bbm}
\usepackage{color,lscape}

\numberwithin{equation}{section}

\newcommand{\hsp}[1]{{\hbox{\hspace{#1}}}}

\def\a{\alpha}

\def\d{\delta}

\def\z{\zeta}
\def\m{\mu}
\def\n{\nu}
\def\s{\sigma}

\def\w{\omega}

\def\tAd{\mathrm{Ad}} 
 
\def\tAut{\mathrm{Aut}}

\def\fb{\mathfrak{b}} 

\def\bC{\mathbb C}

 \def\tcpt{\mathrm{cpt}}
\def\cD{\mathcal D}

 \def\tdim{\mathrm{dim}}

\def\tEnd{\mathrm{End}} 
 \def\fe{\mathfrak{e}}

\def\cF{\mathcal F} 
\def\ff{\mathfrak{f}}

\def\tFlag{\mathrm{Flag}}

\def\sG{\mathscr{G}}
\def\cG{\mathcal G} 

\def\tGr{\mathrm{Gr}}
\def\fg{{\mathfrak{g}}}

\def\tHom{\mathrm{Hom}}
 
\def\bh{\mathbf{h}}
\def\bi{\mathbf{i}}

\def\cL{\mathcal L}

\def\cM{\mathcal M}

\def\tmax{\mathrm{max}}

\def\bP{\mathbb P} 
\def\sP{\mathscr{P}} 
 
\def\fp{\mathfrak{p}}

\def\bQ{\mathbb Q}

\def\bR{\mathbb R}

\def\trank{\mathrm{rank}}

 \def\sfs{\mathsf{s}}
\def\tSL{\mathrm{SL}} \def\tSO{\mathrm{SO}}
\def\tSp{\mathrm{Sp}} 
 
 \def\tSym{\mathrm{Sym}}

\def\fsl{\mathfrak{sl}} \def\fso{\mathfrak{so}} 
\def\fsp{\mathfrak{sp}} 
  
 \def\ttT{\mathtt{T}}
\def\ft{\mathfrak{t}}

 \def\sV{\mathscr{V}}

 \def\bZ{\mathbb Z}

\def\half{\tfrac{1}{2}}

\def\third{\tfrac{1}{3}}

\def\one{\mathbbm{1}}
\def\tand{\quad\hbox{and}\quad}

\def\dfn{\stackrel{\hbox{\tiny{dfn}}}{=}}
\def\sbullet{{\hbox{\tiny{$\bullet$}}}}

\def\inj{\hookrightarrow}

\def\op{\oplus}
\def\ot{\otimes}

\def\tw{\hbox{\small $\bigwedge$}}

\def\wtL{{\Lambda_\mathrm{wt}}}
\def\rtL{{\Lambda_\mathrm{rt}}}
\newcounter{numcnt}

\newcounter{cnt}

\newcounter{acnt}

\newenvironment{a_list}{ 
  \begin{list}{{\emph{(\alph{acnt})}}}
   {\usecounter{acnt} \setlength{\itemsep}{3pt}
    \setlength{\leftmargin}{25pt} \setlength{\labelwidth}{20pt} }
   }
   {\end{list}}
\newcounter{Acnt}

\newcounter{icnt}

\newenvironment{i_list_nem}{ 
  \begin{list}{{(\roman{icnt})}}
   {\usecounter{icnt} \setlength{\itemsep}{3pt}
    \setlength{\leftmargin}{25pt} \setlength{\labelwidth}{20pt} }
   }
   {\end{list}}
\newcounter{Icnt}

\newcounter{exam_cnt}

\newcounter{mccnt}

\newenvironment{bcirclist}{ 
  \begin{list}{\boldmath$\circ$\unboldmath}
   {\usecounter{cnt} \setlength{\itemsep}{2pt}
    \setlength{\leftmargin}{15pt} \setlength{\labelwidth}{20pt} }
   }
   {\end{list}}

\newtheorem{corollary}[equation]{Corollary}
\newtheorem*{corollary*}{Corollary}
\newtheorem{lemma}[equation]{Lemma}
\newtheorem*{lemma*}{Lemma}
\newtheorem{proposition}[equation]{Proposition}
\newtheorem*{proposition*}{Proposition}
\newtheorem{theorem}[equation]{Theorem}
\newtheorem*{theorem*}{Theorem}

\theoremstyle{remark}
\newtheorem*{assume*}{Assume}

\newtheorem*{claim*}{Claim}

\newtheorem{definition}[equation]{Definition}
\newtheorem*{definition*}{Definition}
\newtheorem{example}[equation]{Example}
\newtheorem*{example*}{Example}
\newtheorem*{hint*}{Hint}
\newtheorem*{notation*}{Notation}
\newtheorem*{question*}{Question}
\newtheorem*{answer*}{Answer}
\newtheorem{remark}[equation]{Remark}
\newtheorem*{remark*}{Remark}
\newtheorem*{remarks*}{Remarks}
\newtheorem*{fact*}{Fact}

\numberwithin{HWeq}{section}
\theoremstyle{definition}

\theoremstyle{remark}

\begin{document}
\title[Principal Hodge representations]{Principal Hodge representations}
\author[Robles]{C. Robles}
\email{robles@math.tamu.edu}
\address{Mathematics Department, Mail-stop 3368, Texas A\&M University, College Station, TX  77843-3368} 
\thanks{Robles is partially supported by NSF DMS-1006353.}
\date{\today}
\begin{abstract}
We study Hodge representations of absolutely simple $\bQ$--algebraic groups with Hodge numbers $\bh = (1,1,\ldots,1)$.  For those groups that are not of type A, we give a classification of the $\bR$--irreducible representations; a similar classification for type A does not seem possible.
\end{abstract}
\maketitle
\setcounter{tocdepth}{1}

\section{Introduction}

This paper is a study of Hodge structures with Hodge numbers $\bh = (1,1,\ldots,1)$.  Examples of such Hodge structures include the cohomology group $H^1(X,\bQ)$ of an elliptic curve, the cohomology group $H^3(X,\bQ)$ of a mirror quintic variety \cite{MR2457736}, and the weight component $W_6 H^6(X,\bQ)$ of the cohomology group of a general fibre of a family of quasi-projective $6$-folds studied by Dettweiler and Reiter \cite{MR2660679, KP2012}.  More precisely, the article addresses the question: what are the Hodge representations (\S\ref{S:HRdfn}) with Hodge numbers $\bh = (1,1,\ldots,1)$?  To simplify exposition, we call such representations `principal,' cf. Remark \ref{R:princ}.  The principal Hodge representations are necessarily of Calabi-Yau type; Calabi-Yau Hodge representations of (absolutely) simple Hodge groups are classified in \cite[\S B]{schubVHS}.  I emphasize that the classification of principal Hodge representations is essentially a representation theoretic question; I will not directly address the Hodge theory of complex algebraic varieties.

\subsection*{Motivation}

Very roughly, the classification question is motived by the question: Does a Hodge domain admit a distinguished realization as a Mumford--Tate domain.  A more precise statement of the question, and its relation to this work, is given in Section \ref{S:context}?

\subsection*{Contents}
The main results of this paper are characterizations of the principal Hodge representations $(V,\varphi)$ of a $\bQ$--algebraic, absolutely simple\footnote{$G$ is absolutely simple if $G_\bC$ is simple.} group $G$ in the case that (i) the vector space $V_\bR$ is an irreducible $G_\bR$--module, and (ii) the Lie algebra $\fg_\bC \not= \fsl_{r+1}\bC$.  The principal Hodge representations of symplectic groups (type C) are identified in Theorem \ref{T:C}, those of orthogonal groups (types B and D) are identified in Theorems \ref{T:B} and \ref{T:D}, and those of the exceptional groups of type $G_2$ are identified in Theorem \ref{T:G2}.  The exceptional groups of types $E_r$ and $F_4$ admit no such principal Hodge representations (Corollary \ref{C:E8F4} and Theorem \ref{T:E}).  (For the exceptional groups, we may drop the assumption (ii) that $V_\bR$ is irreducible.)

I am unable to give a similar characterization in the case that $\fg_\bC = \fsl_{r+1}\bC$ (type A).  A comparison of the theorems above with the results and examples of \S\ref{S:A} suggest that no such characterization is available in this case: given a representation $\pi : G \to \tAut(V,Q)$, the circle $\varphi$ giving $\pi$ the structure of a principal Hodge representation is essentially unique (if it exists) for types B, C, D and G; in contrast, the circle is far from unique in type A, cf. \S\ref{S:Aeg}.  Nonetheless, we may make some general observations (cf. \S\ref{S:Arest}); for example, if $G$ admits a principal Hodge representation (with $V_\bR$ not necessarily irreducible), then the rank $r$ is necessarily odd (Lemma \ref{L:r+1}).  The rank one, three and five cases are worked out in \S\ref{S:rank1}, \ref{S:rank3} and \ref{S:rank5}, respectively, and examples of rank seven and nine are considered in \S\ref{S:Aeg}.

\subsection*{Acknowledgements}
I thank P. Griffiths for bringing the classification question to my attention.

\section{Hodge representations}

\subsection{Definition} \label{S:HRdfn}

Let $V$ be a rational vector space, $w \in \bZ$, and let $Q = V \times V \to \bQ$ be a nondegenerate bilinear form satisfying $Q(u,v) = (-1)^wQ(v,u)$, for all $u,v\in V$.  A \emph{polarized Hodge structure of weight $w$ on $V$} is given by a nonconstant homomorphism $\phi : S^1 \to \tAut(V_\bR,Q)$ of $\bR$--algebraic groups with the properties that $\phi(-1) = (-1)^w\one$, and $Q(v,\phi(\bi)\overline v) > 0$ for all $0\not= v \in V_\bC$.  (Here, $\bi = \sqrt{-1}$.)  The associated \emph{Hodge decomposition} $V_\bC = \op_{p+q=w} V^{p,q}$ is given by $V^{p,q} = \{ v \in V_\bC \ | \ \phi(z)v = z^{p-q} v \,,\ \forall \ z\in S^1\}$.  The \emph{Hodge numbers} $\bh = (h^{p,q})$ of $\phi$ are $h^{p,q} = \tdim_\bC\,V^{p,q}$.

Let $G$ be a $\bQ$--algebraic group, and let $\pi : G \to \tAut(V,Q)$ be a homomorphism of $\bQ$--algebraic groups, such that $\pi_*:\fg_\bC \to \tEnd(V,Q)$ is injective.  Let $\varphi : S^1 \to G_\bR$ be a nonconstant homomorphism  of $\bR$--algebraic groups.  Then $(V,Q,\pi,\varphi)$   is a \emph{Hodge representation of $G$} if $\pi\circ\varphi$ is a $Q$--polarized Hodge structure on $V$ of weight $w$.  Any $\bQ$--algebraic group admitting a Hodge representation is a \emph{Hodge group}.  A Hodge group $G$ is necessarily reductive \cite[(I.B.6)]{MR2918237}, and $G_\bR$ contains a compact, maximal torus $T \supset \varphi(S^1)$; that is, $\tdim_\bR(T) = \trank(G_\bR)$, cf. \cite[(IV.A.2)]{MR2918237}.

In general, the bilinear form $Q$ and representation $\pi$ will be dropped from the notation, and the Hodge representation will be denoted by $(V,\varphi)$.

\subsection{Context} \label{S:context}

Hodge representations were introduced by Green, Griffiths and Kerr in \cite{MR2918237} to determine: (i) which reductive, $\bQ$--algebraic groups $G$ admit the structure of a Mumford--Tate group, and (ii) what one can say about various realizations of $G$ as a Mumford--Tate group, and the associated Mumford--Tate domains (which generalize period domains).  Mumford--Tate groups are the symmetry groups of Hodge theory:  the \emph{Mumford--Tate group} $G_\phi \subset \cG = \tAut(V,Q)$ of a polarized Hodge structure $(V,Q,\phi)$ is precisely the $\bQ$--algebraic group preserving the Hodge tensors \cite[(I.B.1)]{MR2918237}.  Moreover, $\phi(S^1) \subset G_\phi(\bR)$, so that $(V,Q,\phi)$ is a Hodge representation of $G_\phi$, and the Mumford--Tate group is a Hodge group.  

Mumford--Tate domains arise as follows.  Let $F^p = \op_{q\ge p} V^{q,w-q}$ denote the Hodge filtration of $V_\bC = V \ot_\bQ \bC$ induced by the Hodge structure $\phi$.  Define Hodge numbers $f^p = \tdim_\bC F^p$ and $\mathbf{f} = (f^\sbullet)$.  Then $F^\sbullet$ is an element of the isotropic flag variety $\tFlag_\mathbf{f}^Q(V_\bC)$, and the latter is a $\cG_\bC = \tAut(V_\bC,Q)$--homogeneous manifold.  (If the weight $w$ is even, then $\cG_\bC = \tSO(V_\bC)$ is an orthogonal group; if $w$ is odd, then $\cG_\bC = \tSp(V_\bC)$ is a symplectic group.)  The \emph{period domain} $\cD \subset \tFlag_\mathbf{f}^Q(V_\bC)$ is the $\cG_\bR = \tAut(V_\bR,Q)$--orbit of $F^\sbullet$.  The \emph{Mumford--Tate domain} $D\subset\cD$ is the $G_\phi(\bR)$--orbit of $F^\sbullet$.  

As a $G_\phi(\bR)$--homogeneous manifold $D \simeq G_\phi(\bR)/R$; here $R$ is the centralizer of the circle $\phi(S^1)$ and contains the compact, maximal torus $T$, cf. \cite[II.A]{MR2918237}.  The homogeneous manifold $G_\phi(\bR)/R$ is a \emph{Hodge domain}.  One may see from this discussion that one subtlety that arises in the subject is the fact that a Hodge domain will admit various realizations as a Mumford--Tate domain (and with inequivalent $G_\phi(\bR)$--homogeneous complex structures).  One would like to know if a given Hodge domain admits a distinguished realization as a Mumford--Tate domain.\footnote{By analogy, the distinguished realization of a rational homogeneous variety $\sG/\sP$ is the unique minimal homogeneous embedding $\sG/\sP \inj \bP\sV$.  The Pl\"ucker embedding of the Grassmannian $\tGr(k,\bC^n) \simeq \tSL(\bC^n)/\sP_k$ is an example.}  The classification of principal Hodge representations is motivated by this question.

A Hodge representation $(V,\phi)$ of $G$ is principal if and only if $f^{p+1} = f^p-1$, cf. Definition \ref{D:principal}.  Equivalently, the period domain $\cD = \tFlag_\mathbf{f}^Q(V_\bC)$ is a full flag variety, and the horizontal distribution (known as Griffiths' transversality or the infinitesimal period relation) on the period domain $\cD$ is bracket--generating.  In this case, the centralizer $R$ of the circle $\phi(S^1)$ is the torus $T$, cf. Remark \ref{R:borel}.  So the principal Hodge representation yields a distinguished realization of the Hodge domain $G_\phi(\bR)/T$ as a Mumford--Tate domain $D \subset \cD$.

\subsection{The grading element \boldmath$\ttT_\varphi$\unboldmath} \label{S:GE}

Fix a Hodge representation $(V,\varphi)$.  To the circle $\varphi(S^1)$ is naturally associated a semisimple $\ttT_\varphi \in \bi\fg_\bR$.  What follows is a terse review of $\ttT_\varphi$; see \cite[\S2.3]{schubVHS} for details.  The semisimple $\ttT_\varphi$ has the property that the $\ttT_\varphi$--eigenvalues of $V_\bC$ are rational numbers (in fact, elements of $\half\bZ$), and the $\ttT_\varphi$--eigenspace decomposition 
$$
  V_\bC \ = \ \bigoplus_{\ell \in \bQ}\, V_{\ell} \,,\quad
  V_\ell \ \dfn \ \{ v \in V_\bC \ | \ \ttT_\varphi(v) = \ell v \}
$$
is the Hodge decomposition: that is, $V^{p,q} = V_\ell$, where $\ell = (p-q)/2\in\half\bZ$.  Additionally, the $\ttT_\varphi$--eigenvalues of $\fg_\bC$ are integers, and the $\ttT_\varphi$--eigenspace decomposition 
$$
  \fg_\bC \ = \ \bigoplus_{\ell\in\bZ}\, \fg_\ell\,,\quad
  \fg_\ell \ \dfn \ \{ \z\in\fg_\bC \ | \ [\ttT^\ell,\z] = \ell\z\} \,,
$$ 
is the Hodge decomposition associated to the weight zero Hodge structure $\tAd\circ\pi$ on $\fg$: that is, $\fg^{\ell,-\ell} = \fg_\ell$.  Moreover, 
$$
  \fg_\ell( V_q) \ \subset \ V_{q+\ell} \tand
  [\fg_k,\fg_\ell] \ \subset \ \fg_{k+\ell} \,.
$$
In particular, the Hodge filtration $F^\sbullet$ of $V_\bC$, defined by $F^p = V_{\ge p}$, is stabilized by the parabolic subalgebra 
\begin{equation} \label{E:p}
  \fp \ \dfn \ \fg_{\ge0} 
\end{equation}
of $\fg_\bC$.  

Fix a compact maximal torus $T \subset G_\bR$ containing the circle $\varphi(S^1)$.  Let $\ft \subset \fg_\bR$ and $\ft_\bC \subset \fg_\bC$ denote the associated Lie algebras.  Then $\ttT_\varphi \in \bi\ft$.  Let $\Delta = \Delta(\fg_\bC,\ft_\bC) \subset \ft_\bC^*$ denote the roots of $\ft_\bC$.  Let $\rtL \subset \wtL \subset \ft_\bC^*$ denote the root and weight lattices.  The connected real groups $G_\bR$ with Lie algebra $\fg_\bR$ are indexed by sub-lattices $\rtL \subset \Lambda \subset \wtL$.  The torus is 
$$
  T \ = \ \ft / \Lambda^* \,,\quad\hbox{where}\quad
  \Lambda^* \ \dfn\ \tHom(\Lambda,2\pi\bi\bZ)\,,
$$
and the weights $\Lambda(V_\bC)$ of $V_\bC$ are contained in $\Lambda$.  (Conversely, if $U$ is an irreducible $\fg_\bC$--module of highest weight $\m \in \Lambda$, then $U$ is a $G_\bC$--module.)  Any $\ttT \in \tHom(\Lambda,\half\bZ)$ determines a circle $\varphi : S^1 \to T$ such that $\ttT_\varphi = \ttT$.

Notice that $\ft_\bC \subset \fg_0 \subset \fp$; so we may select a Borel subalgebra $\fb$ so that $\ft_\bC \subset \fb \subset \fp$.  Define positive roots by 
$$
  \Delta^+ \ \dfn \ \Delta(\fb) \ = \ 
  \{ \a \in \Delta \ | \ \fg^\a \subset \fb \} \,.
$$
Let $\Sigma = \{ \s_1,\ldots,\s_r\}$ denote the simple roots of $\Delta^+$, and let $\{ \ttT^1 , \ldots, \ttT^r\} \subset \tHom(\rtL,\bZ) \subset \ft_\bC$ denote the dual basis of $\ft_\bC$; that is, $\sigma_i(\ttT^j) = \d^j_i$.  The lattice $\tHom(\rtL,\bZ)$ is the \emph{set of grading elements}.  The fact that the $\ttT_\varphi$--eigenvalues of $\fg_\bC$ are integers implies $\ttT_\varphi \in \tHom(\rtL,\bZ)$; that is, $\ttT_\varphi$ is a grading element.  Therefore, $\ttT_\varphi$ is necessarily of the form $n_i \ttT^i$ for some $n_i \in \bZ$.  The condition $\fg^{\s_i} \subset \fb \subset \fp = \fg_{\ge 0}$ is equivalent to 
\begin{equation} \label{E:n>=0}
  \s_i(\ttT_\varphi) \ = \ n_i \ \ge \ 0 \,,
\end{equation}
for all $1\le i\le r$.

\subsection{Real, complex and quaternionic representations}  \label{S:RCQ}
Suppose that $V_\bR$ is an irreducible $G_\bR$--module.  Then there exists an irreducible $G_\bC$--module $U$ such that one of the following holds.
\begin{bcirclist}
\item As a $G_\bR$--module $U$ is \emph{real}: $V_\bC = U \simeq U^*$.
\item As a $G_\bR$--module $U$ is \emph{quaternionic}: $V_\bC = U \op U^*$ and $U \simeq U^*$.
\item As a $G_\bR$--module $U$ is \emph{complex}: $V_\bC = U \op U^*$ and $U \not\simeq U^*$.
\end{bcirclist}
Notice that $U$ is complex if and only if $U$ is not self-dual.  The real and quaternionic cases are distinguished as follows.  Define 
\begin{equation} \label{E:Tcpt}
  \ttT^\tcpt \ \dfn \ 2 \sum_{n_i \in 2\bZ} \ttT^i \,.
\end{equation}
Let $\m \in \Lambda(U) \subset \Lambda$ be the highest weight of $U$.  If $U \simeq U^*$, then 
\begin{equation} \label{E:RQtest}
  \hbox{\emph{$U$ is real (resp., quaternionic) if and only if 
  $\m(\ttT^\tcpt)$ is even (resp., odd),}}
\end{equation}
cf. \cite[(IV.E.4)]{MR2918237}.  

\subsection{Principal Hodge representations} \label{S:PHR}
Let 
$$
  m \ \dfn \ \tmax\{ q \in \bQ \ | \ V_q \not=0 \} \ \in\ \half \bZ \,.
$$
Then the $\ttT_\varphi$--eigenspace decomposition of $V_\bC$ is 
$$
  V_\bC \ = \ V_m \,\op\, V_{m-1}\,\op\cdots\op\,V_{1-m}\,\op\,V_{-m}\,,
$$
and the Hodge numbers are $\bh = (h_m,h_{m-1},\ldots,h_{1-m},h_{-m})$, where 
$h_\ell = \tdim_\bC V_\ell$.

\begin{definition} \label{D:principal}
The Hodge representation $(V,\varphi)$ is \emph{principal} if the Hodge numbers are $\bh = (1,1,\ldots,1,1)$; that is $h_\ell = 1$ for all $-m \le \ell \le m$.
\end{definition}

\begin{remark} \label{R:princ}
The Hodge representation $(V,\varphi)$ is principle if and only if $\ttT_\varphi$ is the mono-semisimple element of a principal $\fsl_2\bC \subset \fsl(V_\bC)$.  See \cite{MR0114875}.
\end{remark}

If $(V,\varphi)$ is principal, then it is necessarily the case that 
\begin{eqnarray}
  \label{E:Tmf} &
  \hbox{\emph{the $\ttT_\varphi$--eigenvalues of $V_\bC$ 
       are multiplicity-free,}} 
  & \\
  \label{E:m_dim}
  & \hbox{and}\quad
  m \ = \ \half \left( \tdim_\bC V_\bC \,-\, 1 \right) \,. &
\end{eqnarray}
The eigenvalues are determined by the weights $\Lambda(V_\bC) \subset \ft_\bC^*$ of $V_\bC$.  Precisely, the $\ttT_\varphi$--eigenvalues of $V_\bC$ are $\{ \lambda(\ttT_\varphi) \ | \ \lambda \in \Lambda(V_\bC) \}$.  In particular, $(V,\varphi)$ is principal if and only if
\begin{equation} \label{E:ev}
  \{ \lambda(\ttT_\varphi) \ | \ \lambda \in \Lambda(V_\bC) \}
  \ = \ \{ m \,,\, m-1 \,,\, m-2 \,, \ldots ,\, 2-m \,,\, 1-m \,,\, -m \} 
\end{equation}
and \eqref{E:m_dim} holds.  

Note that the $\ttT_\varphi$--eigenvalues of $V_\bC$ are multiplicity-free only if 
\begin{equation}\label{E:mf}
  \hbox{\emph{$V_\bC$ is weight multiplicity-free.}}
\end{equation}
Recall that $\Lambda(U^*) = -\Lambda(U)$.  So, if $V_\bC = U \op U^*$, then $\Lambda(V_\bC) = \Lambda(U) \cup -\Lambda(U)$.  The necessary condition \eqref{E:mf} implies
\begin{eqnarray}
  \label{E:RC} &
  \hbox{\emph{The $\fg_\bC$--module $U$ is either real or complex.}} & \\
  \label{E:cpx} &
  \hbox{\emph{If $U$ is complex, then $\Lambda(U) \cap -\Lambda(U) = \emptyset$.}} &
\end{eqnarray}

\subsection{Multiplicity-free representations} \label{S:mf}

Suppose that $(V,\varphi)$ is a principal Hodge representation, and that $V^0 \subset V_\bR$ is an irreducible $G_\bR$--submodule.  Let $U$ be the corresponding irreducible $\fg_\bC$--module, cf. \S\ref{S:RCQ}.  By \eqref{E:mf}, it is necessarily the case that 
\begin{equation}\label{E:Umf}
  \hbox{\emph{$U$ is weight multiplicity-free.}}
\end{equation}  
The weight multiplicity-free representations have been classified by \cite[Theorem 4.6.3]{MR1321638}, in the case that $\fg_\bC$ is simple.  Let $\w_1,\ldots,\w_r$ denote the fundamental weights of $\fg_\bC$.

\begin{theorem} \label{T:mf}
Let $\fg_\bC$ be a simple, complex Lie group.  The irreducible, weight multiplicity-free representations $U$ of $\fg_\bC$, with highest weight $\m$, are as follows:
\begin{a_list}
  \item If $\fg_\bC = \fsl_{r+1}\bC$, then $U$ is one of
   $\tw^k\bC^{r+1}$, and $\m = \w_k$;
   $\tSym^a\bC^{r+1}$, and $\m = a\,\w_1$; or
   $\tSym^a(\bC^{r+1})^*$, and $\m = a\,\w_r$.
  \item If $\fg_\bC = \fso_{2r+1}\bC$, then $U$ is either the standard representation 
    $\bC^{2r+1}$, with $\m = \w_1$; or the spin representation, with $\m = \w_r$. 
  \item If $\fg_\bC = \fsp_{2r}\bC$, then either $U$ is the standard representation 
    $\bC^{2r}$, with $\m = \w_1$; 
    or $r \in \{2,3\}$ and $U \subset \tw^r\bC^{2r}$, with $\m = \w_r$.
  \item If $\fg_\bC = \fso_{2r}\bC$, then $U$ is either the standard representation 
    $\bC^{2r}$, with $\m = \w_1$; or one of the spin representations, 
    with $\m = \w_{r-1},\w_r$.
  \item If $\fg_\bC$ is the exceptional Lie algebra $\fe_6(\bC)$, 
  then $\m = \w_1,\w_6$ and $\tdim_\bC U = 27$; 
  if $\fg_\bC = \fe_7(\bC)$, then $\mu = \w_7$ and $\tdim_\bC U = 56$.
  If $\fg_\bC$ is the exceptional Lie algebra $\fg_2(\bC)$, then $U = \bC^7$ and $\m=\w_1$.
  There are no weight multiplicity-free representations for the exceptional Lie algebras 
  $\fe_8$ and $\ff_4$.
\end{a_list}
\end{theorem}

From \eqref{E:Umf} and Theorem \ref{T:mf}(e) we obtain the following.

\begin{corollary} \label{C:E8F4}
The exceptional Lie groups of type $E_8$ and $F_4$ admit no principal Hodge representations.
\end{corollary}

\subsection{Weights} \label{S:wts}

It will be helpful to review some well-known properties of the weights $\Lambda(U)$ of the irreducible $\fg_\bC$--module $U$.  Assume $\fg_\bC$ is semi-simple, and let $\m$ denote the highest weight of $U$.  Given any weight $\lambda \in \Lambda(U)$, there exists an (ordered) sequence $\{ \s_{i_1},\ldots,\s_{i_\ell}\}$ of simple roots such that
\begin{i_list_nem}
\item $\lambda = \m - (\s_{i_1}+\cdots+\s_{i_\ell})$, and 
\item $\m - (\s_{i_1}+\cdots+\s_{i_j}) \in \Lambda(U)$, for all $1 \le j \le \ell$.
\end{i_list_nem}
By (i), any weight $\lambda$ of $U$ is of the form $\m - \lambda^i\s_i$, for some $0 \le \lambda^i \in \bZ$.  So, \eqref{E:n>=0} implies $\m(\ttT_\varphi)$ is the largest eigenvalue of $U$:
\begin{equation} \label{E:max}
  \lambda(\ttT_\varphi) \ \le \ \m(\ttT_\varphi) \quad\hbox{for all } \ 
  \lambda \in \Lambda(U) \,.
\end{equation}

If $U^\m \subset U$ is the highest weight line, then the weight space of $\lambda = \m- (\s_{i_1}+\cdots+\s_{i_\ell})$ is $U^\lambda = \fg^{-\s_{i_\ell}}\cdots \fg^{-\s_{i_1}}\cdot U^\m$.  In particular, since $\pi_* : \fg_\bC \to \tEnd(V,Q)$ is injective, for each $1 \le i \le r$, there exists $\lambda \in \Lambda(U)$ such that $\lambda -\s_i \in \Lambda(U)$.  Therefore, both $\lambda(\ttT_\varphi)$ and $\lambda(\ttT_\varphi) - n_i$ are eigenvalues of $U$.  In particular, $\ttT_\varphi$--eigenvalues of $V_\bC$ are multiplicity-free only if $n_i \not=0$ for all $1\le i\le r$.  By \eqref{E:n>=0}, $n_i \ge0$.  Therefore, if the Hodge representation $(V,\varphi)$ is principal, it is necessarily the case that 
\begin{equation} \label{E:n>0}
  n_i \ > \ 0 \,,
\end{equation}
for all $1 \le i\le r$.  

\begin{remark} \label{R:borel}
It follows from this discussion and \eqref{E:p} that the stabilizer $\fp = \fg_{\ge0}$ of the Hodge filtration $F^\sbullet$ in $\fg_\bC$ is the Borel $\fb$.  In particular, the Hodge domain is $G_\phi(\bR)/T$, with $T$ a compact, maximal torus.
\end{remark}

Let $\m^*$ denote the highest weight of $U^*$.  Swapping $U$ and $U^*$ if necessary, we may assume that 
\begin{equation} \label{E:m*<m}
  \m(\ttT_\varphi) \ \le \m^*(\ttT_\varphi) \,.
\end{equation}
Then \eqref{E:ev} and \eqref{E:max} yield
\begin{equation} \label{E:mmu}
  m \ = \ \m(\ttT_\varphi) \,.
\end{equation}
Define
\begin{equation} \label{E:m*}
  m^* \ \dfn \ \m^*(\ttT_\varphi) \ \le \ m \,.
\end{equation}
By \eqref{E:ev}, any two $\ttT_\varphi$--eigenvalues of $V_\bC$ differ by an integer; therefore
$$
  0 \le m - m^* \ \in \ \bZ \,.
$$
If $U$ is complex, so that $V_\bC = U \op U^*$ then \eqref{E:Tmf} and \eqref{E:m*<m} force
$$
  m^* \ < \ m \,.
$$

The notation 
$$
  (\lambda^1\cdots\lambda^r) \ \dfn \ \lambda = \m - \lambda^i\s_i
$$
will be convenient.  For example, $\m$ is denoted $(0\cdots0)$.  Set
$$
  |\lambda| \ = \ \sum \lambda^i \,.
$$
The weights $\{ \lambda \in \Lambda(U) \ : \ |\lambda|=1\}$ may be described very easily: let $\m = \m^i \w_i$; then, $\m - \s_i$ is a weight of $U$ if and only if $\m^i \not=0$, cf. \cite[Proposition 3.2.5]{MR2532439}.  If $\fg_\bC$ is simple, and $U$ is weight multiplicity-free, then $\m = p \w_i$, by Theorem \ref{T:mf}.  Therefore, the two largest $\ttT_\varphi$--eigenvalues of $U$ are $m = \m(\ttT_\varphi)$ and $m-n_i = (\m-\s_i)(\ttT_\varphi)$.  The following lemma is a consequence of this discussion and equations \eqref{E:Tmf} and \eqref{E:ev}.

\begin{lemma} \label{L:iv}
Assume that $\fg_\bC$ is simple.  Suppose that $(V,\varphi)$ is a principal Hodge representation of $G$, and that $V_\bR$ is an irreducible $G_\bR$--submodule.  Let $U$ be the associated irreducible, weight multiplicity-free $\fg_\bC$--module (cf. \S\ref{S:RCQ}) of highest weight $\m = p \w_i$, and satisfying \eqref{E:m*<m}.
\begin{a_list}
\item
If $U$ is real ($V_\bC = U$), then $n_i=1$.
\item
Suppose $U$ is complex ($V_\bC = U \op U^*$ and $U \not= U^*$).  Then $0 < m - m^* \in \bZ$.  Define $1 \le i^* \le r$ by $\m^* = p\w_{i^*}$.  Then $m^* = m-1$ if and only if $n_i > 1$; and $m^* = m-1$ and $n_{i^*}=1$ if and only if $n_i > 2$.
\end{a_list}
\end{lemma}

\section{Symplectic Hodge groups}

\begin{theorem} \label{T:C}
Let $G$ be a Hodge group with complex Lie algebra $\fg_\bC = \fsp_{2r}\bC$.  Assume that $(V,\varphi)$ is a Hodge representation with the property that $V_\bR$ is an irreducible $G_\bR$--module.  Let $\ttT_\varphi$ be the associated grading element \emph{(\S\ref{S:GE})}, and assume the normalization \eqref{E:n>=0} holds.  Then the Hodge representation is principal if and only if one of the following holds:
\begin{a_list}
\item $V_\bC = \bC^{2r}$ is the standard representation, and $\ttT_\varphi = \ttT^1 + \cdots + \ttT^r$; 
\item $r=2$, $V_\bC \subset \tw^2\bC^4$ is the irreducible $G_\bC$--module of highest weight $\w_2$, and $\ttT = \ttT^1 + \ttT^2$; 
\item $r=3$, $V_\bC \subset \tw^3\bC^6$ is the irreducible $G_\bC$--module of highest weight $\w_3$, and $\ttT_\varphi = 3 \ttT^1 + \ttT^2 + \ttT^1$.
\end{a_list}
\end{theorem}

\begin{proof}
The representations of $G_\bC$ are self-dual.  So the representation $U$ associated to $V_\bR$ (cf. \S\ref{S:RCQ}) is either real or quaternionic.  By \eqref{E:RC}, if $(V,\varphi)$ is principal, then $U$ is necessarily real, so that $V_\bC = U$.  The condition \eqref{E:Umf} and Theorem \ref{T:mf}(c) restrict our attention to the case that $U$ is one of the the following fundamental representations:  either $U_{\w_1} = \bC^{2r}$ of highest weight
$$
  \w_1 \ = \ \left(\s_1 + \cdots + \s_{r-1} + \half \s_r \right) \,,
$$
for arbitrary $r$; or $U_{\w_r} = \tw^r \bC^{2r}$ of highest weight 
\begin{eqnarray*}
  \w_2 \ = \ \s_1 + \s_2 & & \hbox{if } r=2 \,,\\
  \w_3 \ = \ \s_1 + 2 \s_2 + \tfrac{3}{2} \s_3 & & \hbox{if } r=3 \, .
\end{eqnarray*}

\medskip

\noindent{\small{\bf (A)}}  Let's begin with the case that $\m = \w_1$.  The weights of $\bC^{2r}$ are 
\begin{eqnarray*}
  \Lambda(\bC^{2r}) & = & \{ \w_1 \} \,\cup\,
  \{ \w_1-(\s_1+\cdots+\s_i) \ | \ 1 \le i \le r \}\\
  & & \ \cup \
  \{ \w - (\s_1+\cdots+\s_{i-1}) - 2(\s_i+\cdots+\s_{r-1}) - \s_r \ | 
  \ 1 \le i \le r-1 \} \,.
\end{eqnarray*}
It is straightforward to check that $\ttT_\varphi = \ttT^1 + \cdots + \ttT^r$ is the unique grading element satisfying the normalization \eqref{E:n>=0}, and yielding the consecutive, multiplicity-free eigenvalues \eqref{E:ev}.  

\medskip

\noindent{\small{\bf (B)}}  Next, consider the case that $r=2$ and $\m = \w_2$.  Here $\tdim_\bC U = 5$ and the weights of $U$ are
\begin{eqnarray*}
  \Lambda(U) & = & \{ \pm(\s_1+\s_2) \,,\, \pm\s_1 \,,\, 0 \} \\
  & = & \{ \w_2\,,\, \w_2 - \s_2 \,,\, \w_2-(\s_1+\s_2) \,,\, 
           \w_2-(2\s_1+\s_2) \,,\, \w_2-2(\s_1+\s_2) \} \,.
\end{eqnarray*}
It is clear that $\ttT = \ttT^1 + \ttT^2$ is the unique element satisfying the normalization \eqref{E:n>=0}, and yielding the consecutive, multiplicity-free eigenvalues \eqref{E:ev}.  

\medskip

\noindent{\small{\bf (C)}}  Finally, suppose that $r=3$ and $\m = \w_3$.  Then $\tdim_\bC U = 14$, and the weights of $U$ are
\begin{eqnarray*}
  \Lambda(U) & = & 
  \{ (000) \,,\, (001) \,,\, (011) \,,\, (021) \,,\, (022) \,,\, 
  (111) \,,\, (121) \,,\\
   & & \hsp{7pt} (122) \,,\, (132) \,,\, (221) \,,\, (222) \,,\, (232) \,,\, 
   (242) \,,\, (243) \} \,.
\end{eqnarray*}
Above, we utilize the notation $(\lambda^1\lambda^2\lambda^3) = \w_3 - (\lambda^1\s_1+\lambda^2\s_2+\lambda^3\s_3)$ introduced in \S\ref{S:wts}.  In particular, the weights of $U$ include $\{\w_3 \,,\, \w_3 - \s_3 \,,\, \w_3-(\s_2+\s_3) \}$.  Since the remaining weights are all of the form $\w_3 - (a\s_1+b\s_2+c\s_3)$, with $0\le a,b,c\in\bZ$ and $a+b+c \ge 3$, this forces $1=n_3=n_2$.  

The conditions \eqref{E:m_dim} and \eqref{E:mmu} yield $13/2 = \m(\ttT_\varphi) = \m(n_1 \ttT^1 + \ttT^2+\ttT^3)$, so that $n_1 = 3$.  Thus, the grading element is necessarily of the form given in (c).  Given $\ttT_\varphi = 3\ttT^1 + \ttT^2 + \ttT^3$, is straight-forward to compute $\Lambda(U)(\ttT_\varphi) = \{ 13/2 , 11/2 , \ldots , -11/2 , -13/2 \}$.  Thus \eqref{E:ev} holds.

\medskip

Finally, observe that in each of the cases {\small{(A)}}, {\small{(B)}} and {\small{(C)}} above, \eqref{E:Tcpt} yields $\ttT^\tcpt=0$, so that $\m(\ttT^\tcpt)=0$ and $V_\bC = U$ is real by \eqref{E:RQtest}, as required.
\end{proof}

\section{Orthogonal Hodge groups}

\begin{theorem} \label{T:B}
Let $G$ be a $\bQ$--algebraic group with complex Lie algebra $\fg_\bC = \fso_{2r+1}\bC$.  Assume that $(V,\varphi)$ is a Hodge representation of $G$ with the property that $V_\bR$ is an irreducible $G_\bR$--module.  Let $\ttT_\varphi$ be the associated grading element \emph{(\S\ref{S:GE})}, and assume the normalization \eqref{E:n>=0} holds.  Then the Hodge representation is principal if and only if one of the following holds:
\begin{a_list}
\item $V_\bC = \bC^{2r+1}$ is the standard representation, and $\ttT_\varphi = \ttT^1 + \cdots + \ttT^r$;
\item $V_\bC = U_{\w_r}$ is the spin representation, 
\begin{equation} \label{E:B}
    \ttT_\varphi \ = \ 2^{r-2}\,\ttT^1 \, + \, 2^{r-3}\,\ttT^2 \, + \, 
    2^{r-4}\,\ttT^3 \,+\cdots \, + 2\,\ttT^{r-2}  \, + \, \ttT^{r-1} 
    \,+\, \ttT^r \,,
\end{equation}
and $(r-2)(r-1) \in 4 \bZ$.
\end{a_list}
\end{theorem}

\begin{theorem} \label{T:D}
Let $G$ be a Hodge group with complex Lie algebra $\fg_\bC = \fso_{2r}\bC$.  Assume that $(V,\varphi)$ is a Hodge representation with the property that $V_\bR$ is an irreducible $G_\bR$--module.  Let $\ttT_\varphi$ be the associated grading element \emph{(\S\ref{S:GE})}, and assume the normalization \eqref{E:n>=0} holds.  Then the Hodge representation is principal if and only if one of the following holds:
\begin{a_list}
\item $r\ge4$ is even, $V_\bC = U$ is a spin representation ($\m = \w_{r-1},\w_r$), 
\begin{equation}\label{E:Da}
  \ttT_\varphi \ = \ 2^{r-3}\,\ttT^1 \, + \, 2^{r-4}\,\ttT^2 \, + \, 
  2^{r-5}\,\ttT^3 \,+\cdots\, + 2\,\ttT^{r-3}  \, + \, \ttT^{r-2} \,+\, 
  \ttT^{r-1} \,+\, \ttT^r \,,
\end{equation}
and $(r-3)(r-2) \in 4 \bZ$;
\item $r \ge 5$ is odd, $V_\bC = U \op U^*$, where $U \not\simeq U^*$ is a spin representation ($\m = \w_{r-1},\w_r$) and $\ttT_\varphi$ is one of 
\begin{equation} \label{E:Db}
\begin{array}{rcl}
  \ttT_\varphi & = &  2^{r-2}\,\ttT^1 \, + \, 2^{r-3}\,\ttT^2 \,+\cdots\, 
  + 4\,\ttT^{r-3}  \, + \, 2\,\ttT^{r-2} \,+\, \ttT^{r-1} \,+\, 3\,\ttT^r \,,\\
  \ttT_\varphi & = &  2^{r-2}\,\ttT^1 \, + \, 2^{r-3}\,\ttT^2 \,+\cdots\, 
  + 4\,\ttT^{r-3}  \, + \, 2\,\ttT^{r-2} \,+\, 3\,\ttT^{r-1} \,+\, \ttT^r \,.
\end{array}
\end{equation}
\end{a_list}
\end{theorem}

\subsection*{Proof of Theorem \ref{T:B}}

Let $U$ be the irreducible $\fso_{2r+1}\bC$--module associated to $V_\bR$ as in \S\ref{S:RCQ}.  All representations of $\fso_{2r+1}\bC$ are self-dual; therefore, if $(V,\varphi)$ is principal, it is necessarily the case that $U$ is a real representation of $G_\bR$, by \eqref{E:RC}, and $V_\bC = U$.  By \eqref{E:Umf}, $U$ is weight multiplicity-free.  The irreducible, multiplicity-free representations of $\fso_{2r+1}\bC$ are given by Theorem \ref{T:mf}(b): either $U = \bC^{2r+1}$ is the standard representation, with highest weight 
$$
  \w_1 \ = \ \left(\s_1 \,+\,\s_2\,+ \cdots +\, \s_r \right) \,;
$$
or $U$ is the spin representation of highest weight
$$
  \w_r \ = \ \half \left( \s_1 \,+\, 2\,\s_2 \,+\cdots+\,r\,\s_r \right) \,.
$$
We consider each case below.

\subsubsection*{The standard representation}

Suppose that $U = \bC^{2r+1}$ is the standard representation.  The weights of $U$ are 
\begin{eqnarray*}
  \Lambda(V_\bC) & = & \{ \w_1 \} \,\cup\,\{ \w_1-(\s_1+\cdots+\s_i) \ | \ 1 \le i \le r \}\\
  & & \ \cup \
  \{ \w_1 - (\s_1+\cdots+\s_{i-1}) + 2(\s_i+\cdots+\s_r) \ |  \ 1 \le i \le r \} \,.
\end{eqnarray*}
It is straight forward to confirm that $\ttT_\varphi = \ttT^1+\cdots+\ttT^r$ is the unique grading element satisfying the normalization \eqref{E:n>=0} and such that the $\ttT_\varphi$--eigenvalues of $U$ satisfy \eqref{E:m_dim} and \eqref{E:ev}. 

In particular, $\ttT^\tcpt = 0$, by \eqref{E:Tcpt}, so that $U$ is real, by \eqref{E:RQtest}, as required by \eqref{E:RC}.

\subsubsection*{The spin representation: preliminaries}

Suppose that $U$ is the spin representation.  Then $\m = \w_r$.  The weights of $U$ are parameterized by 
\begin{subequations} \label{SE:spinoddwts}
\begin{equation}
  \cL(U) \ = \ \{ \lambda = (\lambda^1,\ldots,\lambda^r) \in \bZ^r \ | \ 
  \lambda^1 \in \{0,1\} \,,\ \hbox{and }
  \lambda^{i} - \lambda^{i-1} \in \{0,1\} \,,\ \forall \ 1 < i \le r \} \, ;
\end{equation}
specifically, 
\begin{equation}
  \Lambda(U) \ = \ \{ \w_r - \lambda^i\s_i \ | \ \lambda \in \cL(U) \} \,.
\end{equation}
\end{subequations}
It will be convenient to define (i) a filtration $\{0\} = \cF^0 \subset \cF^{1} \subset \cF^{2} \subset \cdots \subset \cF^{r-1} \subset \cF^r = \cL(U)$,
\begin{subequations} \label{SE:spin_filt_dcmp}
\begin{equation} \label{E:spin_filt}
  \cF^s \ \dfn \ 
  \{ \lambda \in \cL(U) \ | \ 0 = \lambda^1,\ldots,\lambda^{r-s} \} \,, \quad
  0 \le s \le r\, ;
\end{equation}
and (ii) a decomposition $\cL(U) = \cL^0\sqcup\cL^1\,\sqcup\cdots\sqcup\,\cL^r$, 
\begin{equation} \label{E:spin_dcmp}
  \cL^s \ \dfn \ \cF^s \backslash \cF^{s-1} 
  \ \stackrel{\eqref{SE:spinoddwts}}{=} \
  \{ \lambda \in \cL(U) \ | \ \lambda^{r-s} = 0 \,,\ \lambda^{r-s+1}=1 \} \,,
  \quad 0 \le s \le r\,,
\end{equation}
\end{subequations}
with the convention that $\cF^{-1} = \emptyset$. 

The spin representation is of dimension $2^r$.  By \eqref{E:m_dim} and \eqref{E:mmu},
\begin{equation} \label{E:spinoddm}
  \half(2^r-1) \ = \ m \ = \   \w_r(\ttT_\varphi)\,.
\end{equation}

\subsubsection*{The spin representation: examples}

Before continuing with the proof, we consider some examples.

\begin{example} \label{eg:spinodd2}
If $r = 2$, then $\tdim_\bC U = 4$, and the weights $\lambda \in \cL$ are 
$$
  (00) \,,\ (01) \,,\ (11) \,,\ (12) \, .
$$
This forces $\ttT_\varphi = \ttT^1 + \ttT^2$.  The $\ttT_\varphi$--graded decomposition of $U$ is $U_{3/2} \op U_{1/2} \op U_{-1/2} \op U_{-3/2}$.  Thus \eqref{E:m_dim} and \eqref{E:ev} hold.  Moreover, $\ttT^\tcpt = 0$, so that $\w_r(\ttT^\tcpt)$ is even; thus, $U$ is real.
\end{example}

\begin{example} \label{eg:spinodd3}
If $r = 3$, then $\tdim_\bC U = 8$, and the weights are 
$$
  (000) \,,\ (001) \,,\ 
  \begin{array}{c} (011) \\ (012) \end{array} \,,\ 
  \begin{array}{cc} (111) & (112) \\ (122) & (123) \end{array} \, .
$$
This forces $\ttT_\varphi = 2\ttT^1 + \ttT^2 + \ttT^3$.  For this grading element, $\w_r(\ttT_\varphi) = \half 7$, as required by \eqref{E:spinoddm}, and the eigenvalues $\{ \pm \half p \ | \ p=1,3,5,7 \}$ of $\ttT_\varphi$ are all multiplicity free.  Moreover, $\ttT^\tcpt = 2\ttT^1$, so that $\w_r(\ttT^\tcpt)=1$ is odd; thus, $U$ is quaternionic.
\end{example}

\begin{example} \label{eg:spinodd4}
If $r = 4$, then $\tdim_\bC U = 16$, and the weights are 
$$
  (0000) \,,\ (0001) \,,\ 
  \begin{array}{c} (0011) \\ (0012) \end{array} \,,\ 
  \begin{array}{c} (0111) \\ (0112) \\ (0122) \\ (0123) \end{array} \,,\
  \begin{array}{cc} (1111) & (1112) \\ (1122) & (1222) \\
    (1123) & (1223) \\ (1233) & (1234)
  \end{array} \, .
$$
So, in order to have multiplicity-free, $\ttT_\varphi$--eigenvalues, we must have $n_4 = n_3 = 1$, $n_2 = 2$ and $n_1 = 4$.  This choice does indeed give consecutive eigenvalues $\{ \pm\half p \ | \ p=1,3,5,\ldots,15\}$, and in particular satisfies $\w_4(\ttT_\varphi) = \half( 4 + 2\cdot2 + 3\cdot1 + 4 \cdot 1) = \half 15$, as required by \eqref{E:spinoddm}.    Moreover, $\ttT^\tcpt = 2(\ttT^1+\ttT^2)$, so that $\w_r(\ttT^\tcpt)=1 + 2 = 3$ is odd; thus, $U$ is quaternionic.
\end{example}

\begin{example} \label{eg:spinodd5}
If $r = 5$, then $\tdim_\bC U = 32$, and the weights are 
$$
  \begin{array}{c} (00000) \\ (00001) \\ (00011) \\ (00012) \end{array} \,,\ 
  \begin{array}{c} (00111) \\ (00112) \\ (00122) \\ (00123) \end{array} \,,\
  \begin{array}{cc} (01111) & (01112) \\ (01122) & (01222) \\
    (01123) & (01223) \\ (01233) & (01234)
  \end{array} \,,\
  \begin{array}{cccc}
     (11111) & (11112) & (11122) & (11123) \\
     (11222) & (11223) & (12222) & (12223) \\ 
     (11233) & (12233) & (11234) & (12333) \\ 
     (12234) & (12334) & (12344) & (12345) 
  \end{array} \, .
$$
In order for $\ttT_\varphi$ to have the consecutive, multiplicity-free eigenvalues required by \eqref{E:ev}, we must have $\ttT_\varphi = 8\ttT^1 + 4 \ttT^2 + 2 \ttT^3 + \ttT^4 + \ttT^5$.  For this grading element, \eqref{E:m_dim} holds.  Moreover, $\ttT^\tcpt = 2(\ttT^1+\ttT^2 + \ttT^3)$, so that $\w_r(\ttT^\tcpt)=1+2+3$ is even; thus, $U$ is real.
\end{example}

\subsubsection*{The spin representation: completing the proof}

We now return to the proof of Theorem \ref{T:B}.  Observe that 
$$
  \cF^{2} \ = \ 
  \left\{ (0\cdots0)\,,\ (0\cdots01)\,,\
     (0\cdots011)\,, \ (0\cdots012) \right\} \,.
$$  
The $\ttT_\varphi$--eigenvalues for these weights are 
$$
  \cF^{2}(\ttT_\varphi) \ = \ \left\{ m \,,\ m-n_r\,,\ 
  m-(n_{r-1}+n_r) \,,\
  m-(n_{r-1}+2n_r) \right\}\,.
$$
By \eqref{SE:spinoddwts}, all other weights of the representation satisfy
\begin{equation} \label{E:p1} 
  \lambda(\ttT_\varphi) \ \le \ m - (n_{r-2} + n_{r-1} + n_r)
  \quad\hbox{for all}\quad \lambda \not\in\cF^{2} \,.
\end{equation}
Therefore, the requirement \eqref{E:ev} forces 
$$
  n_r \ = \ n_{r-1} \ = \ 1 \,.
$$
These first four eigenvalues are 
\begin{equation} \label{E:cF2}
  \cF^{2}(\ttT_\varphi) \ = \ \{ m - p \ | \ p=0,1,2,3\}\,.
\end{equation}

As noted in \eqref{E:p1}, the largest $\ttT_\varphi$--eigenvalue $\lambda(\ttT_\varphi)$ among the $\lambda \not\in \cF^{2} $ is $m- (n_{r-2}+n_{r-1}+n_r) = m - 2 - n_{r-2}$, and is realized by the weight $\lambda = (0\cdots0111)$.  Given \eqref{E:cF2}, the requirement \eqref{E:ev} then forces
$$
  n_{r-2} \ = \ 2 \,.
$$
Note that $\cL^3 = \{(0\cdots0111)\,,\ (0\cdots0112)\,,\ (0\cdots0122) \,,\ (0\cdots0123)\}$, and the eigenvalues 
$$
  \cL^{3}(\ttT_\varphi) \ = \ \{ m-p \ | \ p=4,5,6,7\} \tand
  \cF^{3}(\ttT_\varphi) \ = \ \{ m-p \ | \ p=0,1,\ldots,7\} \, .
$$

We will establish \eqref{E:B} by induction.  We will need the formula
\begin{equation} \label{E:2t}
  2^{t+1}  \ = \ 
  1 \,+\,1 \,+\, 2 \,+\, 4 \,+\, 8 \,+\cdots+\,2^{t-1} \,+\, 2^t \,, 
\end{equation}
which is easily confirmed by an inductive argument.  Suppose that there exists $2 \le \sfs \le r-2$ such that
\begin{subequations}\label{SE:soIH}
\begin{eqnarray}
  \label{E:soIH1}
  n_{r - t} & = & 2^{t-1} \tand \\
  \label{E:soIH2}
  \cL^{t+1}(\ttT_\varphi) & = & 
  \{ m - p \ | \ p = 2^t , 2^t+1 , \ldots , 2^{t+1}-1 \}\,,
\end{eqnarray}
\end{subequations}
for all $2\le t \le \sfs$.  In the preceding paragraph, we saw that \eqref{SE:soIH} holds for $\sfs = 2$.  To complete the proof of the lemma we need to show that
\begin{subequations}\label{SE:soInd}
\begin{eqnarray}
  \label{E:soInd1}
  n_{r-\sfs-1} & = &  2^{\sfs} \tand \\
  \label{E:soInd2}
  \cL^{\sfs+2}(\ttT_\varphi) & = & \{ m - p \ | \ 
  p = 2^{\sfs+1} , 2^{\sfs+1}+1 , \ldots , 2^{\sfs+2}-1 \} \,.
\end{eqnarray}
\end{subequations}
Note that \eqref{E:soIH2} implies that the eigenvalues 
\begin{equation}\label{E:so1}
 \cF^{\sfs+1}(\ttT_\varphi)
  \ = \ \{ m - p \ | \ p = 0 , 1,2, \ldots , 2^{\sfs+1}-1 \}\,.
\end{equation}
By \eqref{SE:spinoddwts} and \eqref{E:soIH2}, 
\begin{equation} \label{E:maxFodd}
  m - (n_{r-s-1} + n_{r-s} + \cdots + n_r) \ = \ 
  \tmax\{ \lambda(\ttT_\varphi) \ | \ \lambda\not\in \cF^{s+1} \}
\end{equation}
is realized by the weight $\lambda = (0^{r-s-2}\,1^{s+2})$.  By \eqref{E:2t},  \eqref{E:soIH1} an \eqref{E:maxFodd} the largest eigenvalue amongst the $\{ \lambda(\ttT_\varphi) \ | \ \lambda \not\in \cF^{\sfs+1} \}$ is $m - (n_{r-\sfs-1} + 2^{\sfs-1} + 2^{\sfs-2} + \cdots + 2 + 1 + 1) = m - (n_{r-\sfs-1} + 2^{\sfs})$.  Given \eqref{E:so1} and \eqref{E:ev}, this observation,  forces $n_{r-\sfs-1} + 2^{\sfs} = 2^{\sfs+1}$.  That is, $n_{r-\sfs-1} = 2^{\sfs}$.  This establishes \eqref{E:soInd1}.

It remains to prove \eqref{E:soInd2}.  Given $\lambda \in \cL^{\sfs+2}$, note that \eqref{SE:spinoddwts} and \eqref{E:spin_dcmp} imply either $\lambda^{r-\sfs} = 1$ or $\lambda^{r-\sfs}=2$.  In particular, we have a disjoint union
$$
  \cL^{\sfs+2} \ = \ \cL^{\sfs+2}_1 \,\sqcup\, \cL^{\sfs+2}_2\,,
$$
given by 
$$
  \cL^{\sfs+2}_i \ \dfn \ 
  \{ \lambda \in \cL^{\sfs+2} \ | \ \lambda^{r-\sfs} = i\}\,.
$$  

Elements of $\cL^{\sfs+2}_1$ are of the form $\lambda = (0\cdots011\lambda^{r-\sfs+1}\cdots\lambda^r)$.  The map $\lambda \mapsto \lambda + \s_{r-\sfs-1} = (0\cdots001\lambda^{r-\sfs-1}\cdots\lambda^r)$ defines a bijection $\cL^{\sfs+2}_1 \to \cL^{\sfs+1}$.  Given \eqref{E:soIH2}, we have
\begin{subequations} \label{E:oddII}
\begin{equation} 
\renewcommand{\arraystretch}{1.3}
\begin{array}{rcl}
  \cL^{\sfs+2}_1(\ttT_\varphi) & = & 
  \{ \lambda(\ttT_\varphi) - 2^\sfs \ | \ \lambda \in \cL^{\sfs+1} \} \\
  & = & 
  \{ m - p \ | \ p = 2^{\sfs+1} , 2^{\sfs+1}+1 , \ldots , 
  3\cdot2^{\sfs}-1 \} \,.
\end{array}
\end{equation} 
Likewise, elements of $\cL^{\sfs+2}_2$ are of the form $\lambda = (0\cdots012\lambda^{r-\sfs+1}\cdots\lambda^r)$, and the map $\lambda \mapsto \lambda - (\s_{r-\sfs-1} + \s_{r-\sfs} + \cdots + \s_r) = (0\cdots001(\lambda^{r-\sfs+1}-1)\cdots(\lambda^r-1))$ defines a bijection $\cL^{\sfs+2}_2 \to \cL^{\sfs+1}$.  

With \eqref{E:2t} and \eqref{E:soIH2}, this implies that the eigenvalues 
\begin{equation} 
\renewcommand{\arraystretch}{1.3}
\begin{array}{rcl}
  \cL^{\sfs+2}_2(\ttT_\varphi)
  & = & 
  \{ \lambda(\ttT_\varphi) - \ (2^{\sfs} + 2^{\sfs-1} + \cdots + 2 + 1 + 1) 
  \ | \ \lambda\in\cL^{\sfs+1} \} \\
  & = & 
  \{ m - p \ | \ p = 3\cdot2^{\sfs} , 3\cdot2^{\sfs}+1 , \ldots , 2^{\sfs+2}-1 \} \,.
\end{array}
\end{equation}
\end{subequations}
Equations \eqref{E:oddII} establish \eqref{E:soInd2}.  This completes the proof of \eqref{E:B}.

\medskip

By \eqref{E:Tcpt}, $\ttT^\tcpt = 2 (\ttT^1 + \cdots + \ttT^{r-2})$.  Thus, $\m(\ttT^\tcpt) = \w_r(\ttT^\tcpt) = \sum_1^{r-2} i = \half(r-2)(r-1)$.  By \eqref{E:RQtest}, the self-dual $U$ is real, as required by \eqref{E:RC}, if and only if $(r-2)(r-1) \in 4 \bZ$.

\subsection*{Proof of Theorem \ref{T:D}}

Let $U$ be the irreducible $\fso_{2r}\bC$--module associated to $V_\bR$, cf. \S\ref{S:RCQ}.  If $(V,\varphi)$ is principal, $U$ is weight multiplicity-free by \eqref{E:Umf}.  By Theorem \ref{T:mf}(d), $U$ is either the standard representation $\bC^{2r}$ of highest weight
$$
  \w_1 \ = \ \s_1 \,+\cdots+\,\s_{r-2} \ + \ \half \left( \s_{r-1}\,+\,\s_r \right) \,,
$$
or one of the spin representations with highest weight
\begin{eqnarray*}
  \w_{r-1} & = & \half \left( \s_1 \,+\, 2\s_2 \,+\cdots+\,(r-2)\s_{r-2} \right) 
  \ + \ \tfrac{1}{4} \left( r\s_{r-1}\,+\,(r-2)\,\s_r \right) \,,\\
  \w_r & = & \half \left( \s_1 \,+\, 2\,\s_2 \,+\cdots+\,(r-2)\s_{r-2} \right) 
  \ + \ \tfrac{1}{4} \left( (r-2)\s_{r-1}\,+\,r\,\s_r \right)\,.
\end{eqnarray*}

\subsubsection*{The standard representation}

Suppose that $U$ is the standard representation.  Then $U$ is self-dual, and so either real or quaternionic.  By \eqref{E:RC}, if $(V,\varphi)$ is principal, then $U$ is real, so that $V_\bC = U$.  Therefore, the weights of $V_\bC$ are 
\begin{eqnarray*}
  \Lambda(V_\bC) & = & \{ \w_1 \,,\, \w_1-(\s_1+\cdots+\s_{r-2}+\s_r)\} 
  \,\cup\,\{ \w_1-(\s_1+\cdots+\s_i) \ | \ 1 \le i \le r \}\\
  & & \ \cup \
  \{ \w_1 - (\s_1+\cdots+\s_{i-1}) - 2(\s_i+\cdots+\s_{r-2}) - \s_{r-1}-\s_r \ | 
  \ 1 \le i \le r-2 \} \,.
\end{eqnarray*}
It is straight-forward to confirm that 
\begin{equation} \label{E:Dstd}
\begin{array}{rcl}
  \ttT_\varphi & = & \ttT^1+\cdots+\ttT^{r-2} + \ttT^{r-1} + 2\,\ttT^r 
  \quad\hbox{and}\\ 
  \ttT_\varphi & = & \ttT^1+\cdots+\ttT^{r-2} + 2\,\ttT^{r-1} + \ttT^r 
\end{array}
\end{equation}
are the only grading elements satisfying \eqref{E:m_dim} and yielding (multiplicity-free) eigenvalues \eqref{E:ev}.  For these two grading elements we have $\ttT^\tcpt=2\ttT^{r-1}$ and $\ttT^\tcpt=2\ttT^r$, respectively.   Therefore, $\w_1(\ttT^\tcpt)=1$ and $U$ is quaternionic, by \eqref{E:RQtest}, contradicting \eqref{E:RC}.  It follows that there exists no principle Hodge representation $(V,\varphi)$ of $G$ such that $U = \bC^{2r}$. 

\begin{remark} 
The argument above yields the following:  \emph{Let $G$ be a Hodge group with complex Lie algebra $\fg_\bC = \fso_{2r}\bC$.  Assume that $(V,\varphi)$ is a Hodge representation with the property that $V_\bR$ is an irreducible $G_\bR$--module, and the associated $\fg_\bC$--module $U$ (\S\ref{S:RCQ}) is the standard representation.  Let $\ttT_\varphi$ be the associated grading element (\S\ref{S:GE}), and assume the normalization \eqref{E:n>=0} holds.  Then Hodge numbers are $\bh = (2,2,\ldots,2)$ if and only if one of \eqref{E:Dstd} holds.}
\end{remark}

\subsubsection*{The spin representation, $r$ even}

If $r$ is even, then the spin representations are self-dual.  Arguing as above in the case of the standard representation, it is necessarily the case that $r > 4$.  (Else $U$ is quaternionic.)  The two spin representations have highest weight $\m = \w_{r-1}$ and $\m = \w_r$.  The arguments for the two cases are symmetric, so we will assume $\m = \w_r$.  The weights of $U$ are parameterized by 
\begin{subequations} \label{SE:spinevenwts}
\begin{equation} \label{E:spewts1}
\renewcommand{\arraystretch}{1.3}\begin{array}{rcl}
  \cL(U) \ = \ \big\{ \lambda = (\lambda^1,\ldots,\lambda^r) \in \bZ^r 
  & | & 
  \lambda^1,\,\lambda^r +\lambda^{r-1}-\lambda^{r-2},\, 
  \lambda^r-\lambda^{r-1} \in \{0,1\}\,,\\
  & & \lambda^i - \lambda^{i-1} \in \{0,1\} \,,\ \forall \ 1 < i \le r-2 \big\} \,;
\end{array}\end{equation}
specifically, 
\begin{equation}
  \Lambda(U) \ = \ \{ \w_r - \lambda^i\s_i \ | \ \lambda \in \cL(U) \} \,.
\end{equation}
\end{subequations}
Set $m = \w_r(\ttT_\varphi)$.  Given a weight $\lambda \in \cL(U)$, the corresponding $\ttT_\varphi$--eigenvalue is $\lambda(\ttT_\varphi) = m - \sum_i \lambda^i n_i$.
As in the proof of Theorem \ref{T:B}(b), we define a filtration $\cF^{3} \subset \cF^{4} \subset \cdots \subset \cF^{r-1} \subset \cF^r = \cL(U)$
\begin{subequations} \label{SE:spe_filt_dcmp} 
\begin{equation}
  \cF^s \ \dfn \ \{ \lambda \in \cL(U) \ | \ 
  0 = \lambda^1,\ldots,\lambda^{r-s}\}
\end{equation}
and decomposition $\cL(U) = \cL^{3} \sqcup \cL^{4} \sqcup\cdots\sqcup\cL^r$
\begin{equation}
  \cL^s \ \dfn \ \cF^s\backslash\cF^{s-1} \ = \ 
  \{ \lambda \in \cL(U) \ | \ \lambda^{r-s} = 0 \,,\ \lambda^{r-s+1}=1 \} \,,
\end{equation}
for $4 \le s \le r$, and 
\begin{equation}
  \cL^3 \ \dfn \ \cF^3 \ = \ \{ (0\cdots0)\,,\ (0\cdots01) \,,\ (0\cdots0101)
  \,,\ (0\cdots0111) \} \,.
\end{equation}
\end{subequations}
The eigenvalues of $\cL^{3}$ are 
\begin{equation} \label{E:1}
  \cL^{3}(\ttT_\varphi) \ = \ \{ m \,,\, m-n_r \,,\ m-(n_{r-2}+n_r) \,,\ 
  m-(n_{r-2}+n_{r-1}+n_r) \} \, .
\end{equation}
In general, 
\begin{equation} \label{E:2}
  \tmax\{ \lambda(\ttT_\varphi) \ | \ \lambda \not\in \cF^s \} 
  \ = \  m - (n_{r-s} + n_{r-s+1} + \cdots + n_{r-2} + n_r) \,
\end{equation}
is realized by the weight $\lambda = (0^{r-s-1}\,1^{s-1}01)$.  Therefore, $\tmax\{\lambda(\ttT_\varphi) \ | \ \lambda \not\in \cF^3 = \cL^3 \} = m-(n_{r-3} + n_{r-2} + n_r)$.   So, given \eqref{E:1}, the condition \eqref{E:ev} forces
$$
  n_r \ = \ n_{r-2} \ = \ 1 \, .
$$
This updates \eqref{E:1} to $\cL^{3}(\ttT_\varphi) = \{ m \,,\, m-1 \,,\ m-2 \,,\ m-(2+n_{r-1}) \}$.  From $\cL^{4} = \{ (0\cdots01101) \,,\ (0\cdots01111) \,,\ 
(0\cdots01211) \,,\ (0\cdots01212) \}$, we see that the eigenvalues of $\cF^4 = \cL^3 \sqcup \cL^4$ are 
\begin{equation} \label{E:3}
\renewcommand{\arraystretch}{1.3}\begin{array}{rcl}
  \cF^{4}(\ttT_\varphi) & = & \{
  m \,,\, m-1 \,,\ m-2 \,,\ m-(2+n_{r-1}) \,,\ m-(2+n_{r-3}) \,,\\ 
  & & \hsp{7pt} 
  m-(2+n_{r-3}+n_{r-1}) \,, m-(3+n_{r-3}+n_{r-1}) \,,\\ 
  & & \hsp{7pt}  m-(4+n_{r-3}+n_{r-1}) \} \,.
\end{array} \end{equation}
Given \eqref{E:2} and \eqref{E:3}, the condition \eqref{E:ev} implies one of the following two holds
\begin{subequations}
\begin{eqnarray}
  \label{E:op1}
  n_{r-1} = 1 &\hbox{and}& n_{r-3} = 2\,, \quad \hbox{or}\\
  \label{E:op2}
  n_{r-1} = 2 &\hbox{and}& n_{r-3} = 1 \,.
\end{eqnarray}
\end{subequations}
I claim that \eqref{E:op2} is not possible.  For, if \eqref{E:op2} were to hold, then both weights $(0\cdots011212)$, $(0\cdots012211) \in \cL^5$ would have eigenvalue $m-(n_{r-4}+7)$, contradicting the requirement \eqref{E:Tmf} that the eigenvalues have multiplicity one.  Therefore \eqref{E:op1} holds.  The eigenvalues of $\cF^4$ are 
\begin{equation} \nonumber 
  \cF^4(\ttT_\varphi) \ = \ \{ m-p \ | \ p=0,1,\ldots,7\} \,.
\end{equation}

We will establish \eqref{E:Da} by induction.  Suppose that there exists $2 \le \sfs \le r-2$ such that
\begin{subequations}
\begin{eqnarray}
  \label{E:Deveni}
   n_{r - s} & = & 2^{s-2} \quad\hbox{and}\\
  \label{E:Devenii}
  \cF^{s+1}(\ttT_\varphi) & = & \{ m - p \ | \ p = 0,1,\ldots, 2^{s}-1 \} \,,
\end{eqnarray}
\end{subequations}
for all $2 \le s \le \sfs$.  The discussion above establishes these inductive hypotheses for $\sfs = 3$.  To complete the proof of \eqref{E:Da} we need to show that
\begin{subequations}
\begin{eqnarray}
  \label{E:DevenI}
   n_{r-\sfs-1} & = &  2^{\sfs-1} \quad\hbox{and}\\
  \label{E:DevenII}
  \cL^{\sfs+2} & = &  
  \{ m - p \ | \ p = 2^{\sfs} , 2^{\sfs}+1 , \ldots , 2^{\sfs+1}-1 \} \,.
\end{eqnarray}
\end{subequations}
Given \eqref{E:Devenii}, the requirement \eqref{E:ev} that the $\ttT_\varphi$--eigenvalues be consecutive forces the largest eigenvalue $\lambda(\ttT_\varphi)$ among the $\lambda \not\in \cF^{s+1}$ to be $m-2^\sfs$.  By \eqref{E:2}, \eqref{E:Deveni} and \eqref{E:2t}, this largest eigenvalue is $m-n_{r-\sfs-1} + 2^{\sfs-2} + 2^{\sfs-3} + \cdots + 2 + 1 + 1 = m-n_{r-\sfs-1} + 2^{\sfs-1}$.  Therefore, $n_{r-\sfs-1} = 2^{\sfs-1}$, establishing \eqref{E:DevenI}.

To prove \eqref{E:DevenII}, note that \eqref{E:spewts1} implies that 
\begin{subequations}\label{SE:4}
\begin{equation} \label{E:4a}
  \cL^{\sfs+2} \ = \ \cL^{\sfs+2}_1 \,\sqcup\,\cL^{\sfs+2}_2 \,
\end{equation}
where 
$$
  \cL^{\sfs+2}_i \ = \ \{ \lambda \in \cL(U) \ | \ \lambda^{r-\sfs-2} = 0 \,,\
  \lambda^{r-\sfs-1}=1 \,,\ \lambda^{r-\sfs} = i \} \, .
$$
The map $\lambda = (0\cdots011\cdots) \mapsto \lambda + \s_{r-\sfs-1} = (0\cdots001\cdots)$ defines a bijection $\cL^{\sfs+2}_1 \to \cL^{\sfs+1}$.  By \eqref{E:Devenii}, the eigenvalues of $\cL^{\sfs+1} = \cF^{\sfs+1}\backslash\cF^\sfs$ are $\cL^{\sfs+1}(\ttT_\varphi) = \{ m - p \ | \ p=2^{\sfs-1}, 2^{\sfs-1}+1,\ldots,2^\sfs-1 \}$.  So, by \eqref{E:Devenii}, the eigenvalues of $\cL^{\sfs+2}_1$ are 
\begin{equation}
  \cL^{\sfs+2}_1(\ttT_\varphi) \ = \ \{ \lambda(\ttT_\varphi) - n_{r-\sfs-1} \ | \ 
  \lambda \in\cL^{\sfs+1} \} \ = \ 
  \{ m - p \ | \ p=2^\sfs,2^\sfs+1,\ldots,3\cdot 2^{\sfs-1}-1 \} \,.
\end{equation}

Similarly, the map 
$$
  \lambda \ \mapsto \ 
  \left\{ \begin{array}{ll}
  \lambda+\s_{r-\sfs-1}+\cdots+\s_{r-2}+\s_r\,, \ & 
  \hbox{ if } \lambda^r = \lambda^{r-1} +1 \,,\\
  \lambda+\s_{r-\sfs-1}+\cdots+\s_{r-1} \,,\ 
  & \hbox{ if }  \lambda^r=\lambda^{r-1}\,.
  \end{array} \right.
$$
defines a bijection $\cL^{\sfs+2}_2 \to \cL^{\sfs+1}$.  Arguing as above, and making use of \eqref{E:2t}, the eigenvalues of $\cL^{\sfs+2}_2$ are 
\begin{equation} \renewcommand{\arraystretch}{1.3}
\begin{array}{rcl}
  \cL^{\sfs+2}_2(\ttT_\varphi) & = &
   \{ \lambda(\ttT_\varphi) - (2^{\sfs-1}+2^{\sfs-2}+\cdots+2+1+1) \ | \ 
   \lambda \in \cL^{\sfs+1} \} \\
   & = & \{ \lambda(\ttT_\varphi) - 2^\sfs \ | \ \lambda \in \cL^{\sfs+1} \}
   \ = \ \{ m-p \ | \ p=3\cdot2^{\sfs-1},\cdots,2^{\sfs+1}-1 \} \,.
\end{array}
\end{equation}
\end{subequations}
Item \eqref{E:DevenII} now follows from \eqref{SE:4}, and \eqref{E:Da} is established.

\medskip

By \eqref{E:Tcpt}, $\ttT^\tcpt = 2 (\ttT^1 + \cdots + \ttT^{r-3})$.  Thus, $\m(\ttT^\tcpt) = \w_r(\ttT^\tcpt) = \sum_1^{r-3} i = \half(r-3)(r-2)$.  By \eqref{E:RQtest}, the self-dual $U$ is real, as required by \eqref{E:RC}, if and only if $(r-3)(r-2) \in 4 \bZ$.

\subsubsection*{The spin representation, $r$ odd}

If $r$ is odd, then $U^*_{\w_{r-1}} = U_{\w_r}$.  In particular, the spin representations are not self-dual, and $V_\bC = U_{\w_{r-1}} \op U_{\w_r}$.  Without loss of generality, we will assume that the representation $U$ associated to $V_\bR$ in \S\ref{S:RCQ} is $U_{\w_r}$; that is, $\m = \w_r$.  (With the normalization \eqref{E:m*}, this will yield the first grading element of \eqref{E:Db}.  Taking $\m = \w_{r-1}$ yields the second grading element.  The two arguments are symmetric, and we will give only the first.)

The weights of $U$ are given by \eqref{SE:spinevenwts}.  The weights of $U^*$ are parameterized by 
\begin{subequations} \label{SE:spedualwts}
\begin{equation} \label{E:spedwts1}
\renewcommand{\arraystretch}{1.3}\begin{array}{rcl}
  \cL(U^*) \ = \ \big\{ \m=[\m_1,\ldots,\m_r] \in \bZ^r 
  & | & 
  \m^1,\,\m^r +\m^{r-1}-\m^{r-2},\, 
  \m^{r-1}-\m^{r} \in \{0,1\}\,,\\
  & & \m^i - \m^{i-1} \in \{0,1\} \,,\ \forall \ 1 < i \le r-2 \big\} \,;
\end{array}\end{equation}
specifically, 
\begin{equation}
  \Lambda(U^*) \ = \ 
  \{ \w_{r-1} - \m^i\s_i \ | \ \m \in \cL(U^*) \} \,.
\end{equation}
\end{subequations}
Following \eqref{SE:spe_filt_dcmp}, we define a filtration $\cG^{3} \subset \cG^{4} \subset \cdots \subset \cG^{r-1} \subset \cG^r = \cL(U^*)$
\begin{subequations} \label{SE:sped_filt_dcmp} 
\begin{equation}
  \cG^s \ \dfn \ \{ \m \in \cL(U^*) \ | \ 
  0 = \m^1,\ldots,\m^{r-s}\}
\end{equation}
and decomposition $\cL(U^*) = \cM^{3} \sqcup \cM^{4} \sqcup\cdots\sqcup\cM^r$
\begin{equation}
  \cM^s \ \dfn \ \cG^s\backslash\cG^{s-1} \ = \ 
  \{ \m\in \cL(U^*) \ | \ \m^{r-s} = 0 \,,\ \m^{r-s+1}=1 \} \,,
\end{equation}
for $4 \le s \le r$, and 
\begin{equation}
  \cM^3 \ \dfn \ \cG^3 \ = \ \{ [0\cdots0]\,,\ [0\cdots010] \,,\ [0\cdots0110]
  \,,\ [0\cdots0111] \} \,.
\end{equation}
\end{subequations}

I claim that 
\begin{equation} \label{E:seo1}
  n_r \ = \ 3 \tand n_{r-1} \ = \ 1 \, .
\end{equation}
To see this, recall that $m = \w_r(\ttT_\varphi)$, by \eqref{E:mmu}, and $m^* = \w_{r-1}(\ttT_\varphi)$, by \eqref{E:m*}.  By Lemma \ref{L:iv}(b), 
$$
  0 \ < \ m - m^* \ = \ \w_r(\ttT_\varphi) - \w_{r-1}(\ttT_\varphi) 
  \ = \ \half(n_r - n_{r-1}) \ \in \ \bZ \,.
$$ 
With \eqref{E:n>0}, this implies $0 < n_r - n_{r-1} \in 2 \bZ$.  Therefore, $n_r \ge 3$.  By Lemma \ref{L:iv}(b), 
$$
  m - m^* \ = \ 1
$$
and $n_{r-1} = 1$.  Thus, $n_r = 3$. 

The eigenvalues associated to the weights $\cL^3 \cup \cM^3$ are 
\begin{equation} \label{E:LM3}
\renewcommand{\arraystretch}{1.3}
\begin{array}{rcl}
  \cL^3(\ttT_\varphi) & = & \{ m \,,\ m-3\,,\ m-3-n_{r-2} \,,\, m-4-n_{r-2} \}\\
  \cM^3(\ttT_\varphi) & = & \{ m-1 \,,\ m-2 \,,\ m-2-n_{r-2}\,,\ m-5-n_{r-2} \}\,.
\end{array}\end{equation}
If $n_{r-2}=1$, then the eigenvalue $m-3$ will have multiplicity greater than one; so, the condition \eqref{E:mf} that the $\ttT_\varphi$--eigenvalues have multiplicity one forces
$$
  n_{r-2} \ > \ 1 \, .
$$
In analogy with \eqref{E:2},  
\begin{equation} \label{E:5}
\renewcommand{\arraystretch}{1.3}
\begin{array}{rcl}
  \tmax\{ \m(\ttT_\varphi) \ | \ \m \not\in \cG^s\} & = & 
  \w_{r-1}(\ttT_\varphi) - (n_{r-s} + n_{r-s+1} + \cdots + n_{r-1}) \\
  & = &  m - 1 -(n_{r-s} + n_{r-s+1} + \cdots + n_{r-1}) \,
\end{array}\end{equation}
is realized by $\lambda = ( 0^{r-s-1}\,1^{s}0)$.  By \eqref{E:2} and \eqref{E:5} the largest $\ttT_\varphi$--eigenvalues amongst the $\lambda \not\in\cL^3 = \cF^3$ and the $\m \not\in\cM^3 = \cG^3$ are 
\begin{equation} \label{E:maxLM3}
  m - (n_{r-3} + n_{r-2} + 3) \tand
  m - (n_{r-3} + n_{r-2} + 2) \,.
\end{equation}
So, if $n_{r-2} > 2$, then \eqref{E:n>0}, \eqref{E:LM3} and \eqref{E:maxLM3} imply $m-4$ will not appear as an eigenvalue.  This contradicts the requirement \eqref{E:ev} that the $\ttT_\varphi$--eigenvalues be consecutive; thus 
\begin{equation} \label{E:seo2}
  n_{r-2} \ = \ 2 \,.
\end{equation}
From \eqref{E:seo1}, \eqref{E:LM3} and \eqref{E:seo2}, we see that the $\ttT_\varphi$--eigenvalues of $\cL^3 \cup \cM^3$ are 
$$
  \cL^3(\ttT_\varphi) \cup \cM^3(\ttT_\varphi) \ = \ 
  \{ m - p \ | \ p=0,1,\ldots,7\} \,.
$$

We will complete the proof that $\ttT_\varphi$ is the first grading element of \eqref{E:Db} by induction.  Suppose that there exists $2 \le \sfs \le r-2$ such that:
\begin{subequations}
\begin{eqnarray}
 \label{E:Doddi}
 n_{r-s} & = &  2^{s-1} \quad \hbox{and}\\
 \label{E:Doddii}
 \cF^{s+1}(\ttT_\varphi) \cup \cG^{s+1}(\ttT_\varphi) & = &  
 \{ m-p \ | \ p=0,1,\ldots,2^{s+1}-1 \} \,,
\end{eqnarray}
\end{subequations}
for all $2 \le s \le \sfs$.  Keeping in mind that $\cL^3=\cF^3$ and $\cM^3 = \cG^3$, we have seen that this inductive hypothesis holds for $\sfs = 2$.  To complete the induction (and proof of the lemma), we must show that
\begin{subequations}
\begin{eqnarray}
 \label{E:DoddI}
  n_{r-\sfs-1} & = &  2^{\sfs} \quad \hbox{and}\\
 \label{E:DoddII}
 \cL^{\sfs+2}(\ttT_\varphi) \cup \cM^{\sfs+2}(\ttT_\varphi) & = & 
  \{ m - p \ | \ p=2^{\sfs+1},2^{\sfs+1}+1,\ldots, 2^{\sfs+2}-1\} \,.
\end{eqnarray}
\end{subequations}
By \eqref{E:2}, \eqref{E:Doddi} and \eqref{E:2t}, 
\begin{eqnarray*}
  \tmax\{ \lambda(\ttT_\varphi) \ | \ \lambda \not\in \cF^{\sfs+1} \} 
  & = & m-(n_{r-\sfs-1} + 2^{\sfs-1} + \cdots + 4 + 2 + 3) \\
  & = & m - (n_{r-\sfs-1} + 2^{\sfs} + 1) \,.
\end{eqnarray*}  
Similarly, \eqref{E:5}, \eqref{E:Doddi} and \eqref{E:2t} yield 
\begin{eqnarray*}
  \tmax\{ \m(\ttT_\varphi) \ | \ \m \not\in \cG^{\sfs+1} \} 
  & = & m-1-(n_{r-\sfs-1} + 2^{\sfs-1} + \cdots + 4 + 2 + 1) \\
  & = & m - (n_{r-\sfs-1} + 2^{\sfs}) \,.
\end{eqnarray*}
On the other hand \eqref{E:Doddii} and the requirement \eqref{E:ev} that the eigenvalues of $V_\bC$ be consecutive imply $\tmax\{ \lambda(\ttT_\varphi) , \m(\ttT_\varphi) \ | \ \lambda \not\in \cF^{\sfs+1} \,,\ \m \not\in \cG^{\sfs+1}\} = m-2^{\sfs+1}$.  Therefore, $n_{r-\sfs-1} = 2^{\sfs}$, establishing \eqref{E:DoddI}.

In analogy with \eqref{E:4a}, observe that \eqref{E:spedwts1} implies 
\begin{subequations} \label{SE:6}
\begin{equation}
  \cM^{\sfs+2} \ = \ \cM^{\sfs+2}_1 \,\cup\,\cM^{\sfs+2}_2\,,
\end{equation}
where
$$
  \cM^{\sfs+2}_i \ = \ \{ \m \in \cL(U^*) \ | \ \m^{r-\sfs-2} = 0 \,,\ 
  \m^{r-\sfs-1} = 1 \,,\ \m^{r-\sfs} = i \} \,.
$$
The assignment $\n \mapsto \n + \s_{r-\sfs-1}$ defines a bijection $\cL^{\sfs+2}_1 \cup \cM^{\sfs+2}_1 \to \cL^{\sfs+1} \cup \cM^{\sfs+1}$.  The hypothesis \eqref{E:Doddii} implies $\cL^{\sfs+1}(\ttT_\varphi)\cup\cM^{\sfs+1}(\ttT_\varphi) = \{ m - p \ | \ p = 2^\sfs, 2^{\sfs}+1,\ldots, 2^{\sfs+1}-1\}$.  These observations, taken with \eqref{E:DoddI}, yield
\begin{equation}
  \cL^{\sfs+2}_1(\ttT_\varphi) \cup \cM^{\sfs+2}_1(\ttT_\varphi) \ = \ 
  \{ m - p \ | \ p=2^{\sfs+1},2^{\sfs+1}+1,\ldots,3\cdot2^\sfs - 1 \} \, .
\end{equation}
Likewise, the assignment $\n \mapsto \n +  \s_{r-\sfs-1} + 2(\s_{r-\sfs}+\cdots+\s_{r-2}) + \s_{r-1} + \s_r$ defines a bijection $\cL^{\sfs+2}_2 \cup \cM^{\sfs+2}_2 \to \cF^{\sfs} \cup \cG^{\sfs}$.  Taken with \eqref{E:Doddii}, \eqref{E:DoddI} and \eqref{E:2t}, this yields
\begin{equation} \renewcommand{\arraystretch}{1.3}
\begin{array}{rcl}
  \cL^{\sfs+2}_1(\ttT_\varphi) \cup \cM^{\sfs+2}_1(\ttT_\varphi) & = & 
  \{ \n(\ttT_\varphi) - 2^\sfs - 2( 2^{\sfs-1}+\cdots+2)-(1+3)\ | \ 
  \n \in \cF^\sfs \cup \cG^\sfs \} \\
  & = & 
  \{ \n(\ttT_\varphi) - 3\cdot2^\sfs \ | \ 
  \n \in \cF^\sfs \cup \cG^\sfs \} \\
  & = & 
  \{ m-p \ | \ p =  3\cdot2^\sfs , 3\cdot2^\sfs+1,\ldots,2^{\sfs+2}-1\}\,.
\end{array}
\end{equation}
\end{subequations}
Item \eqref{E:DoddII} now follows from \eqref{SE:6}.  This yields (the first grading element of) \eqref{E:Db}.

\section{Exceptional groups}

\begin{theorem} \label{T:E}
Let $G$ be a Hodge group with complex Lie algebra $\fg_\bC = \fe_6(\bC), \fe_7(\bC)$.  Then $G$ does not admit a principal Hodge representation $(V,\varphi)$.
\end{theorem}

\begin{theorem} \label{T:G2}
Let $G$ be a Hodge group with complex Lie algebra $\fg_\bC = \fg_2(\bC)$.  Assume that $(V,\varphi)$ is a Hodge representation.  Let $\ttT_\varphi$ be the associated grading element \emph{(\S\ref{S:GE})}, and assume the normalization \eqref{E:n>=0} holds.  Then the Hodge representation is principal if and only if $V_\bC = \bC^7$ and $\ttT_\varphi = \ttT^1 + \ttT^2$.
\end{theorem}

\subsection*{Proof of Theorem \ref{T:E} for \boldmath$\fg_\bC = \fe_6(\bC)$\unboldmath}

We argue by contradiction.  Suppose a principal Hodge representation exists.  Let $W_1 \op\cdots \op W_s$ be a decomposition of $V_\bR$ into irreducible $G_\bR$--modules.  Let $U_j$ be the irreducible $G_\bC$--module associated to $W_j$, as in \S\ref{S:RCQ}.  By \eqref{E:Umf} and Theorem \ref{T:mf}(e), the highest weight of $U_j$ is either $\w_1$ or $\w_6$.  Moreover, $\w_1^* = \w_6$, so that $U_j$ is complex, $W_j(\bC) = W_j \ot_\bR \bC = U_j \op U_j^*$, and $W_j(\bC) \simeq W_k(\bC)$, for all $j,k$.  The constraint \eqref{E:mf} that $V_\bC = \op_j\, W_j(\bC)$ be weight multiplicity-free forces $s=1$.  Thus, $V_\bR$ is an irreducible $G_\bR$--module, and $V_\bC = U \op U^*$.  

Assume, without loss of generality, that $U$ is the irreducible $\fe_6$ representation of highest weight
$$
  \w_1 \ = \ \third \left( 4 \s_1 \,+\, 3\s_2 \,+\, 5\s_3 \,+\, 6\s_4  \,+\, 
  4\s_5 \,+\, 2\s_6 \right) \,.
$$
Then, $\tdim\,V_\bC = 2\tdim_\bC U = 54$.  Therefore, \eqref{E:m_dim} and \eqref{E:mmu} imply $53/2 = \w_1(\ttT_\varphi) = \third( 4 n_1 + 3n_2 + 5n_3 + 6n_4  + 
4n_5 + 2n_6)$.  Equivalently, 
$$
  53 \cdot 3 \ = \ 
  2 \cdot (4 n_1 \,+\, 3n_2 \,+\, 5n_3 \,+\, 6n_4  \,+\, 4n_5 \,+\, 2n_6) \,.
$$
This is not possible, as $2$ does not divide $53\cdot 3$.

\subsection*{Proof of Theorem \ref{T:E} for \boldmath$\fg_\bC = \fe_7(\bC)$\unboldmath}

We argue by contradiction.  Suppose a principal Hodge representation exists.    Let $W_1 \op\cdots \op W_s$ be a decomposition of $V_\bR$ into irreducible $G_\bR$--modules.  Let $U_j$ be the irreducible $G_\bC$--module associated to $W_j$, as in \S\ref{S:RCQ}.  By \eqref{E:Umf} and Theorem \ref{T:mf}(e), the highest weight of $U_j$ is $\w_7$.  Therefore, $W_j(\bC) \simeq W_k(\bC)$, for all $j,k$.  The constraint \eqref{E:mf} that $V_\bC = \op_j\,W_j(\bC)$ be weight multiplicity-free forces $s=1$, and $V_\bR$ is irreducible.  Moreover, $\w_7^* = \w_7$, so that $U = U_1$ is either real or quaternionic.  By \eqref{E:RC}, $V_\bC = U$ is real.  

The highest weight of $U$ is 
$$
  \w_7 \ = \ \half \left( 2 \s_1 \,+\, 3\s_2 \,+\, 4\s_3 \,+\, 6\s_4  \,+\, 
  5\s_5 \,+\, 4\s_6  \,+\, 3\s_7 \right) \,.
$$
Since $\tdim_\bC U = 56$, \eqref{E:m_dim} yields
\begin{equation}\label{E:em}
  m \ = \ \half 55 \,.
\end{equation}

The weights of $U$ with $|\lambda| \le 4$ are $\w_7 = (0\cdots0)$, 
\begin{equation} \nonumber
  \w_7-\s_7 \,,\quad \w_7-(\s_6+\s_7) \,,\quad \w_7-(\s_5+\s_6+\s_7) \,,\quad \w_7-(\s_4+\s_5+\s_6+\s_7)\,.
\end{equation}
All other weights are of the form $\w_7 - (\s_4+\s_5+\s_6+\s_7) - a^i\s_i$ with $0 \le a^i \in \bZ$.  So, in order for \eqref{E:ev} to hold, it is necessary that
$$ 
  1 \ = \ n_7\,,\ n_6 \,,\ n_5 \,,\ n_4 \,.
$$ 
Observe that we have $\ttT_\varphi$--eigenvalues
$$
  \{ \lambda(\ttT_\varphi) \ | \ |\lambda| \le 4 \} \ = \ 
  \{ m - p \ | \ p=0,1,2,3,4\} \,.
$$

Next, the weights with $|\lambda|=5,6$ are 
\begin{eqnarray}
  \label{E:E7_5}
  \{ \lambda \in \Lambda(U) \ | \ |\lambda|=5\} & = & 
  \{ (0101111)\,,\ (0011111) \} \,,\\
  \label{E:E7_6}
  \{ \lambda \in \Lambda(U) \ | \ |\lambda|=6\} & = & 
  \{ (0111111)\,,\ (1011111) \} \,.
\end{eqnarray}
Therefore, to obtain the eigenvalue $m-5$, as required by \eqref{E:ev}, the weights \eqref{E:E7_5} force either $n_2=1$ and $n_3 > 1$, or $n_2>1$ or $n_3 = 1$.  Before proceeding to consider these two cases, it will be helpful to note that the remaining weights with $\lambda^1=0$ are 
\begin{eqnarray*}
 \{ \lambda \in \Lambda(U) \ | \ \lambda^1=0\,,\ |\lambda|>6 \} & = &  
 \{ (0112111)\,,\ (0112211)\,,\ (0112221)\,,\ (0112222)\}\,.
\end{eqnarray*}

\subsubsection*{Case 1: $n_2=1$ and $n_3>1$}

Then the $\ttT_\varphi$--eigenvalues associated to weights with $\lambda^1=0$ are
$$
  \{ \lambda(\ttT_\varphi) \ | \ \lambda\in\Lambda(U) \,,\ \lambda^1=0 \} \ = \ 
  \{ m - p \ | \ p=0,1,\ldots,5,\, n_3+4\,,\,n_3+5,\ldots,n_3+9 \} \,.
$$
The largest eigenvalue $\lambda(\ttT_\varphi)$ with $\lambda^1\not=0$ is $(1011111)(\ttT_\varphi) = m-(n_1+n_3+4)$.  All other weights with $\lambda^1\not=0$ are $\lambda \le (1\cdots1)$, and so yield an eigenvalue $\lambda(\ttT_\varphi) \le m - (5+n_1+n_3)$.  Therefore, to realize the eigenvalue $m-6$, it is necessary that $n_3=2$.  This then forces $n_1+6 = (1011111)(\ttT_\varphi) = 12$, so that $n_1=6$.  However, if $\ttT_\varphi = n_i\ttT^i = 6\ttT^1 + \ttT^2 + 2\ttT^3 + \ttT^4 + \ttT^5 + \ttT^6 + \ttT^7$, then $\w_7(\ttT_\varphi) = \half( 2\cdot 6 + 3\cdot1 + 4\cdot2 + 6\cdot1+5\cdot1+4\cdot1+3\cdot1) = \half 41 \not= m$, a contradiction of \eqref{E:mmu} and \eqref{E:em}.  

\subsubsection*{Case 2: $n_2>1$ and $n_3=1$}

In this case, we have 
$$
  \{ \lambda(\ttT_\varphi) \ | \ \lambda\in\Lambda(U) \,,\ \lambda^1=0 \} \ = \ 
  \{ m - p \ | \ p=0,1,\ldots,5,\, n_2+4\,,\,n_2+5,\ldots,n_2+9 \} \,.
$$
The largest eigenvalue $\lambda(\ttT_\varphi)$ with $\lambda^1\not=0$ is $(1011111)(\ttT_\varphi) = m-( n_1+5)$.  All other weights with $\lambda^1\not=0$ are $\lambda \le (1\cdots1)$, and so yield an eigenvalue $\lambda(\ttT_\varphi) \le m - (5+n_1+n_2)$.  Therefore, in order to realize the eigenvalue $m-6$, we must have either $n_1=1$ or $n_2 = 2$.

\subsubsection*{Case 2.a: $n_2>1$, $n_3=1$ and $n_1=1$}

If $n_1=1$, then we must have $n_2=3$ in order to realize the eigenvalue $m-7 = (0101111)(\ttT_\varphi)$.  However, if $\ttT = n_i\ttT^i = \ttT^1 + 3\ttT^2 + \ttT^3 + \ttT^4 + \ttT^5 + \ttT^6 + \ttT^7$, then $\w_7(\ttT_\varphi) = \half( 2\cdot 1 + 3\cdot3 + 4\cdot1 + 6\cdot1+5\cdot1+4\cdot1+3\cdot1) = \half 33 \not= m$, a contradiction of \eqref{E:mmu} and \eqref{E:em}.

\subsubsection*{Case 2.b: $n_2=2$ and $n_3=1$}

If $n_2=2$, then in order to avoid multiplicity, as required by \eqref{E:Tmf}, and realize the eigenvalue $m-12$ we must have $n_1 = 7$.  In this case, $\ttT = n_i\ttT^i = 7\ttT^1 + 2\ttT^2 + \ttT^3 + \ttT^4 + \ttT^5 + \ttT^6 + \ttT^7$, so that $\w_7(\ttT_\varphi) = \half( 2\cdot 7 + 3\cdot2 + 4\cdot1 + 6\cdot1+5\cdot1+4\cdot1+3\cdot1) = \half 42 \not= m$.  Again, this contradicts \eqref{E:mmu} and \eqref{E:em}.  

\subsection*{Proof of Theorem \ref{T:G2}}

Suppose $(V,\varphi)$ is a principal Hodge representation of $G$.  Let $W_1 \op\cdots \op W_s$ be a decomposition of $V_\bR$ into irreducible $G_\bR$--modules.  Let $U_j$ be the irreducible $G_\bC$--module associated to $W_j$, as in \S\ref{S:RCQ}.  By \eqref{E:Umf} and Theorem \ref{T:mf}(e), the highest weight of $U_j$ is $\w_1$.  Therefore, $W_j(\bC) \simeq W_k(\bC)$, for all $j,k$.  The constraint \eqref{E:mf} that $V_\bC = \op_j \, W_j(\bC)$ be weight multiplicity-free forces $s=1$.  Therefore, $V_\bR$ is an irreducible $G_\bR$--module.  Moreover, $\w_1^* = \w_1$, so that $U = U_1$ is either real or quaternionic.  By \eqref{E:RC}, $U$ must be real.  Thus, $V_\bC = U = \bC^7$.  

The highest weight is 
$$
  \w_1 \ = \ 2\s_1 \ + \ \s_2 \,.
$$
Since $U$ is self-dual, it is real or quaternionic, cf. \S\ref{S:RCQ}.  If $V$ is principal, then \eqref{E:RC} implies $U$ is real and $V_\bC = U$.  So the weights of $V_\bC$ are 
\begin{eqnarray*}
  \Lambda(U) & = & \{ \w_1 \,,\, \w_1 - \s_1 \,,\, \w_1 - (\s_1+\s_2) \,,\,  
  \w_1 - (2\s_1+\s_2) \,,\\ 
  & & \quad \w_1 - (3\s_1+\s_2) \,,\, \w_1 - (3\s_1+2\s_2) \,,\,
  \w_1 - (4\s_1+2\s_2) \} \,.
\end{eqnarray*}
In particular, the weights include $\w_1 \,,\, \w_1 - \s_1 \,,\, \w_1-(\s_1+\s_2)$, and all other weights are of the form $\w_1 - (a\s_1 + b\s_2)$ with $0 < a,b\in\bZ$ and $a+b >2$.  Given \eqref{E:ev}, this forces $\ttT_\varphi = \ttT^1 + \ttT^2$, which yields a decomposition $V_\bC = V_3 \op V_2 \op V_1 \op V_0 \op V_{-1} \op V_{-2} \op V_{-3}$ with all Hodge numbers equal to one.

Finally, we note that \eqref{E:Tcpt} yields $\ttT^\tcpt = 0$.  So $\m(\ttT^\tcpt) = 0$, and $U$ is real by \eqref{E:RQtest}.

\section{Special linear Hodge groups} \label{S:A}

Assume throughout this section that 
\begin{center}
   \emph{$G$ is a $\bQ$--algebraic group with Lie algebra 
   $\fg_\bC \simeq \fsl_{r+1}\bC$.}
\end{center} 

The fundamental weights of $\fg_\bC = \fsl_{r+1}\bC$ are 
\begin{equation} \label{E:Afw}
\renewcommand{\arraystretch}{1.6}\begin{array}{rcl}
  \w_k & = & \tfrac{r+1-k}{r+1} \,
  \big[ \s_1 + 2\s_2 + \cdots + (k-1) \s_{k-1} \big] \,+\,
  \tfrac{k(r+1-k)}{r+1} \, \s_k \\
  & & + \
  \tfrac{k}{r+1} \,
  \big[ (r-k)\s_{k+1} + (r-k-1)\s_{k+2} + \cdots + \s_r \big] \,,
\end{array}\end{equation}
$1 \le k \le r$.  If $U$ is the irreducible $\fg_\bC$--module of highest weight $\w_k$, then the highest weight of the dual $U^*$ is
\begin{equation} \label{E:Adual}
  \w_k^* \ = \ \w_{r+1-k} \,.
\end{equation}
In particular, in the notation of Lemma \ref{L:iv}, $k^* = r+1 - k$.
\subsection{Restrictions on principal Hodge representations} \label{S:Arest}

In this section we describe some simple numerical constraints on principal Hodge representations.  Among those are the restriction that the rank of $G$ be odd.  The cases of rank $r=1,3,5$ are addressed in Propositions \ref{P:rank1}, \ref{P:rank3} and \ref{P:rank5}; a examples of rank seven and nine are considered in \S\ref{S:Aeg}.

By \eqref{E:Afw}, 
\begin{equation} \label{E:r+1}
  \w_k(\ttT_\varphi) 
  \ \in \ \frac{1}{r+1}\bZ \,.
\end{equation}

\begin{lemma} \label{L:r+1}
Suppose that $G$ admits a principal Hodge representation.  Then $r+1$ is even.
\end{lemma}

\begin{proof}[Proof of Lemma \ref{L:r+1}]
Let $V_\bR = W_1 \op\cdots\op W_s$ be a decomposition into irreducible $G_\bR$--modules.  Let $U_j$ be the irreducible $G_\bC$--module associated to $W_j$ as in \S\ref{S:RCQ}.  If one of the $U_j$ is self-dual, then $r+1$ is necessarily even by \eqref{E:Adual}.   

Assume none of the $U_j$ are self-dual.  Then, by \S\ref{S:RCQ}, $W_j(\bC) = U_j \op U_j^*$, for all $1 \le j \le s$.  In particular, each $W_j(\bC)$ is of even dimension.  It follows that $V_\bC = \op_j\, W_j(\bC)$ is of even dimension.  So \eqref{E:m_dim} implies
\begin{equation} \label{E:L1}
  m \ = \ \frac{2a-1}{2}
\end{equation}
for some $0 < a \in \bZ$.  

Let $\m_j$ denote the highest weight of $U_j$.  By \eqref{E:mmu}, there exists $1 \le j \le s$ such that the largest eigenvalue $m = \m_j(\ttT_\varphi)$.  By \eqref{E:mf}, the representation $U_j$ is weight multiplicity-free.  By Theorem \ref{T:mf}(a), the weight $\m_j$ is necessarily of the form $p \w_k$ for some $p$ and $k$.  From \eqref{E:mmu} and \eqref{E:r+1}, we see that 
\begin{equation} \label{E:L2}
  m \ = \ \frac{b}{r+1}
\end{equation}
for some $0 < b \in \bZ$.  Together \eqref{E:L1} and \eqref{E:L2} yield $(r+1)(2a-1) = 2b$, implying $2$ divides $r+1$.
\end{proof}

\begin{lemma}\label{L:Ar_pw1}
Suppose that $r > 1$ and $G$ admits a principal Hodge representation $(V,\varphi)$ with $V_\bC = \tSym^p(\bC^{r+1}) \op \tSym^p(\bC^{r+1})^*$.  Then $(r+1)! \equiv 0$ mod $2p$, and $2p \not\equiv0$ mod $r+1$.
\end{lemma}

\begin{proof}[Proof of Lemma \ref{L:Ar_pw1}, Part 1]
First, we prove that $(r+1)!\in 2p\,\bZ$.  Note that 
$$
  \tdim_\bC\,\tSym^p(\bC^{r+1}) \ = \ \binom{p+r}{r} \ = \ 
  \frac{(p+r)(p+r-1)\cdots(p+1)}{r!} \,.
$$
The highest weight of $\tSym^p(\bC^{r+1})$ is $\m = p\,\w_1$.  By \eqref{E:m_dim}, \eqref{E:mmu} and \eqref{E:r+1}, 
\begin{eqnarray*}
  \frac{1}{2}\,\left[ 2\, \frac{(p+r)(p+r-1)\cdots(p+1)}{r!} \,-\, 1 \right]
  & \in & \frac{p}{r+1}\,\bZ \,.
\end{eqnarray*}
Equivalently,
\begin{eqnarray*}
 \half (r+1)\,\big[ 2\,(p+r)(p+r-1)\cdots(p+1) \,-\, r! \big]
  & \in & p\,r!\,\bZ \,.
\end{eqnarray*}
Note that $(p+r)(p+r-1)\cdots(p+1) \equiv r!$ mod $p$.  In particular, there exists $q\in\bZ$ such that 
\begin{eqnarray*}
 \half (r+1)\,\big( 2\,r! \,+\, pq \,-\, r! \big) \ = \ 
 \half (r+1)! \,+\, \half(r+1)pq 
  & \in & p\,r!\,\bZ \,.
\end{eqnarray*}
so that $(r+1)!\in 2p\,\bZ$.
\end{proof}

\begin{proof}[Proof of Lemma \ref{L:Ar_pw1}, Part 2]
Now we prove that $2p \not\equiv0$ modulo $r+1$.  The weights of $\bC^{r+1}$ and $U = \tSym^p(\bC^{r+1})$, as a $\fg_\bC$--modules, are 
\begin{eqnarray}
  \nonumber
  \Lambda(\bC^{r+1}) & = &  \left\{
    \nu_i \ \dfn \ \w_1 - (\s_1 + \cdots + \s_{i-1}) \ | \ 1 \le i \le r+1
  \right\} \,,\\
  \label{E:Ar_symwts}
  \Lambda(U) & = & \{ \nu_{i_1} + \cdots + \nu_{i_p} \ | 
  \ 1 \le i_1 \le\cdots \le i_p \le r+1 \} \,.
\end{eqnarray}
Note that 
\begin{equation} \label{E:lsum0}
  \nu_1 \,+\cdots+\, \nu_{r+1} \ = \ 0 \, .
\end{equation} 
By Lemma \ref{L:r+1}, $r+1 = 2s$, for some $s \in \bZ$.  Then \eqref{E:lsum0} implies
\begin{equation} \label{E:lsum-}
  \nu_1 \,+\cdots+\,\nu_s \ = \ -(\nu_{s+1} \,+\cdots+\, \nu_{r+1}) \, .
\end{equation}   

Arguing by contradiction, suppose that $r+1$ divides $2p$.  We will consider two cases: $2p = 2t \,(r+1)$ and $2p = (2t+1)\cdot(r+1)$, for some $t \in \bZ$.
\begin{bcirclist}
\item First, suppose that $2p = 2t\,(r+1)$.  Then $p = t(r+1)$, and \eqref{E:Ar_symwts} and \eqref{E:lsum0} imply $0 = t( \nu_1 + \cdots + \nu_{r+1} ) \in \Lambda(U)$.  This contradicts \eqref{E:cpx}.  
\item
Next, suppose that $2p = (2t+1)\cdot(r+1)$, so that $p = (2t+1)\,s$.  Then \eqref{E:Ar_symwts}, \eqref{E:lsum0} and \eqref{E:lsum-} imply that both
\begin{eqnarray*}
  \nu_1 + \cdots + \nu_s & = & (\nu_1+\cdots+\nu_s) \ + \ 
  t\,( \nu_1 + \cdots + \nu_{r+1}) \,, \quad\hbox{and} \\
  -(\nu_1 + \cdots + \nu_s) & = & (\nu_{s+1}+\cdots+\nu_{r+1}) \\
  & = & (\nu_{s+1}+\cdots+\nu_{r+1}) \ + \ 
  \, t\,( \nu_1 + \cdots + \nu_{r+1}) \,  
\end{eqnarray*}
are weights of $U$.  Again, this contradicts \eqref{E:cpx}.
\end{bcirclist}
\end{proof}

\begin{lemma} \label{L:Ar_w2k}
Suppose that $1 \le k < \half(r+1)$ and that $2k \not= \half(r+1)$.  Assume
$G$ admits a principal Hodge representation $(V,\varphi)$ with $V_\bR$ irreducible and $V_\bC= \tw^{2k}(\bC^{r+1}) \op \tw^{2k}(\bC^{r+1})^*$.  Then $r+1 \equiv 0$ mod $4$.
\end{lemma}

\begin{remark*}
From \eqref{E:Adual}, we see that the condition $2k \not=\half(r+1)$ implies $U = \tw^{2k}(\bC^{r+1})$ is not self-dual.  By \S\ref{S:RCQ}, the representation $V_\bC = U \op U^*$ is complex.
\end{remark*}

\begin{proof}
We argue by contradiction.  Let $\ttT_\varphi = n_i\ttT^i$ be the grading element associated to the Hodge representation.  We have $\tdim_\bC V_\bC = 2a$ for $a = \tdim_\bC U$.  By Lemma \ref{L:r+1}, $r + 1 = 2s$ for some $s \in \bZ$.  By \eqref{E:m_dim} and \eqref{E:mmu}, 
\begin{eqnarray*}
  \half (2a-1) & = &  \w_{2k}(\ttT_\varphi) \\
  & \stackrel{\eqref{E:Afw}}{=} &  \frac{(s-k)}{s} \,
  \left( n_1 \,+\, 2 n_2 \,+\cdots+\, (2k-1)n_{2k-1} \right) \
  + \ \frac{2k\,(s-k)}{s} \, n_{2k} \\\ 
  & &  + \ 
  \frac{k}{s} \, \big(
    (r-2k)\,n_{2k+1} \ + \ (r-2k-1)\,n_{2k+2} \ + \cdots + \ n_r
  \big) \, .
\end{eqnarray*}
Multiplying through by $2s$ yields $(2a-1)s \in 2\,\bZ$.  Thus $r+1 = 2s \equiv 0$ mod 4.
\end{proof}

\subsection{The rank one case} \label{S:rank1}

The irreducible representations of $\fsl_2\bC$ are $\tSym^p\bC^2$; they are of highest weight $\m = p \w_1 = \half p \s_1$.

\begin{proposition} \label{P:rank1}
Let $G$ be a Hodge group with complex Lie algebra $\fg_\bC = \fsl_2\bC$.  Assume that $(V,\varphi)$ is a Hodge representation with the property that $V_\bC = \tSym^p\bC^2$, and satisfying the normalization \eqref{E:n>=0}.  Then $V$ is principal if and only if $\ttT_\varphi = \ttT_1$.
\end{proposition}

\begin{proof}
Let $U$ be the irreducible $G_\bC$--module associated to $V_\bR$ by \S\ref{S:RCQ}.  Let $\m = p \w_1$ be the highest weight of $U$.  By Theorem \ref{T:mf}(a), $U = \tSym^p\bC^2$ is multiplicity-free as required by \eqref{E:Umf}.

Since $U$ is self-dual, it is either real or quaternionic, cf. \S\ref{S:RCQ}.  By \eqref{E:RC}, if $V$ is principal, then $U$ is real; therefore, $V_\bC = U$.
By Lemma \ref{L:iv}(a), $n_1=1$; so the grading element is necessarily of the form $\ttT_\varphi = \ttT^1$.  

The weights of $V_\bC$ are $\Lambda(V_\bC) = \{ \half\,(p-i)\,\s_1 \ | \ 0 \le i \le p \}$.  In particular, \eqref{E:ev} holds.  To conclude that the Hodge representation is principal, it remains to confirm that $V_\bC$ is a real (rather than quaternionic) $G_\bR$--module: definition \eqref{E:Tcpt} yields $\ttT^\tcpt = 0$, so that $\m(\ttT^\tcpt) = 0$ is even, as required by \eqref{E:RQtest} and \eqref{E:RC}.
\end{proof}

\subsection{The rank three case} \label{S:rank3}

\begin{proposition} \label{P:rank3}
Let $G$ be a Hodge group with complex Lie algebra $\fg_\bC = \fsl_4\bC$.  Assume that $(V,\varphi)$ is a Hodge representation with the property that $V_\bR$ is an irreducible $G_\bR$--module.  Let $\ttT_\varphi$ be the associated grading element \emph{(\S\ref{S:GE})}, and assume the normalization \eqref{E:n>=0} holds.  Then the Hodge representation is principal if and only if $V_\bC = \bC^4 \op (\bC^4)^*$ and $\ttT_\varphi$ is one of
$$
  3\ttT^1 \,+\, 2 \ttT^2 \,+\, \ttT^1 \quad\hbox{ or }\quad
  \ttT^1 \,+\, 2\ttT^2 \,+\, 3 \ttT^3 \,.
$$
\end{proposition}

\begin{proof}
Let $U$ be the irreducible $G_\bC$--module of highest weight $\m$ associated to $V_\bR$ as in \S\ref{S:RCQ}.  By \eqref{E:Umf}, $U$ is necessarily weight multiplicity-free.  By Theorem \ref{T:mf}(a), the highest weight of $U$ is either $\m = p \w_1$ (or the symmetric case $\m=p \w_3$); or $\m=\w_2$.

\medskip  

\noindent{\small{\bf (I)}}  Let's begin with the case that $\m = p \w_1$.  Then $U = \tSym^p\bC^4$ and $V_\bC = U \op U^*$.  We have $r+1 = 4 = 2^2$ and $(r+1)! = 24 = 3 \cdot 2^3$.  By Lemma \ref{L:Ar_pw1}, $p | 3 \cdot 2^2$, but $2 \not= p$.  Therefore, $p=1,3$.  

Observe that 
$$
  m \ \stackrel{\eqref{E:mmu}}{=} \ p\w_1(\ttT_\varphi) \ = \ \tfrac{1}{4}( 3 n_1 + 2 n_2 + n_3 ) \,,
  \tand
  0 \ \stackrel{(\ast)}{<} \ m-m^* \ = \ p\half(n_1 - n_3) \,,
$$
where the inequality $(\ast)$ is due to Lemma \ref{L:iv}(b).  From \eqref{E:ev} we see that any two eigenvalues differ by an integer; therefore, $0<m-m^*\in\bZ$.  Since $p=1,3$, this implies $n_1 - n_3 \in 2 \bZ$.  Then \eqref{E:n>0} yields $1 \le n_3 \le n_1-2$.  In particular, $n_1 \ge 3$.  By Lemma \ref{L:iv}(b), we must have $n_3 = 1$ and $m-m^* = 1$, the latter yielding $p=1$ and $n_1=3$.  From $\tdim V_\bC = 8$ and \eqref{E:m_dim} we see that $\half 7 = m$, yielding $n_2 = 2$.  Thus, $\ttT_\varphi= 3\ttT^1 + 2 \ttT^2 + \ttT^1$ is the first grading element of the proposition.

\smallskip

If, on the other hand, we take $\m = p \w_3$, then a similar argument again yields $p=1$ and $\ttT_\varphi$ is the second grading element of the proposition.  In either case ($\m=\w_1$ or $\m = \w_3$), it is easily checked that \eqref{E:ev} holds, using the fact that the weights of $V_\bC$ are 
\begin{eqnarray*}
  \{ \w_1 \,,\ \w_1 - \s_1 \,,\ \w_1 - (\s_1+\s_2) \,,\
  \w_1 - (\s_1+\s_2+\s_3) \} & \cup &\\
  \{ \w_3 \,,\ \w_3 - \s_3 \,,\ \w_3 - (\s_3+\s_2) \,,\
  \w_3 - (\s_3+\s_2+\s_1) \} & & 
\end{eqnarray*}

\medskip  

\noindent{\small{\bf (II)}}  Next, consider the case that $\m = \w_2$.  We have $U = \tw^2\bC^6$.  Since $U$ is self-dual, $U$ is either real or quaternionic (\S\ref{S:RCQ}).  By \eqref{E:RC}, if $V$ is principal, then $U$ is necessarily real, so that $V_\bC = U$.  By Lemma \ref{L:iv}(a), $n_2 = 1$.  The highest weight of the irreducible $V_\bC$ is $\w_2 = \half(\s_1+2\s_2+\s_3)$.  By \eqref{E:m_dim} and \eqref{E:mmu}, we have 
$$
  \half 5 \ = \ \half( n_1 + 2n_2 + n_3 ) = 1 + \half(n_1+n_3) \,,
$$
so that $n_1 + n_3 = 3$.  From \eqref{E:n>0}, we see that either $n_1=1$ and $n_3=2$, or $n_1=1$ and $n_3=2$.   Thus,
\begin{equation} \label{E:ext5}
  \ttT_\varphi \ = \ \ttT^1 \,+\, \ttT^2 \,+\, 2\,\ttT^3 \quad\hbox{or}\quad
  \ttT_\varphi \ = \ 2\,\ttT^1 \,+\, \ttT^2 \,+\, \ttT^3 \,.
\end{equation}
For the first grading element, we have $\ttT^\tcpt = 2 \ttT^3$, cf. \eqref{E:Tcpt}.  So $\m(\ttT^\tcpt) = \w_2(2\ttT^3) = 1$ implying $U$ is quaternionic, a contradiction of \eqref{E:RC}.  Similarly, for the second grading element, $\ttT^\tcpt = 2 \ttT^1$.  So $\m(\ttT^\tcpt) = \w_2(2\ttT^1) = 1$ and $U$ is quaternionic, again contradicting \eqref{E:RC}.  Therefore, there exists no principal Hodge representation $(V,\varphi)$ with $V = \tw^2\bC^4$.
\end{proof}

\begin{remark}
With respect to part (II) of the proof above, observe that the weights of $\tw^2\bC^4$ are 
$$
  (000) \,,\ (010) \,,\ (110) \,,\ (011) \,,\ (111)\,,\ (121) \,.
$$
(Above, we utilize the notation $(\lambda^1\lambda^2\lambda^3) = \w_2 - \lambda^i \s_i$ introduced in \S\ref{S:wts}.)  From this, it is easily checked that both the grading elements \eqref{E:ext5} yield eigenspace decompositions of $\tw^2\bC^4$ with multiplicity-free eigenvalues $\pm\{ \half , \half 3 , \half 5 \}$.  Thus, the following holds:  \emph{Let $G$ be a Hodge group with complex Lie algebra $\fg_\bC = \fsl_4\bC$.  Assume that $(V,\varphi)$ is a Hodge representation with the property that $V_\bR$ is irreducible and $V_\bC = (\tw^2\bC^4) \op (\tw^2\bC^4)^*$.  (That is, $U = \tw^2\bC^4$ is quaternionic.)  Then the Hodge numbers are $\bh = (2,2,\ldots,2,2)$ if only if $\ttT_\varphi$ is one of \eqref{E:ext5}.}
\end{remark}

\subsection{The rank five case} \label{S:rank5}

\begin{proposition} \label{P:rank5}
Let $G$ be a Hodge group with complex Lie algebra $\fg_\bC = \fsl_6\bC$.  Assume that $(V,\varphi)$ is a Hodge representation with the property that $V_\bR$ is an irreducible $G_\bR$--module.  Let $\ttT_\varphi$ be the associated grading element \emph{(\S\ref{S:GE})}, and assume the normalization \eqref{E:n>=0} holds.  Then the Hodge representation is principal if and only if either: \emph{\bf (a)} $V_\bC = \bC^6 \op (\bC^6)^*$ and 
\begin{subequations}
\begin{equation} \label{E:A5_w1}
  \begin{array}{rcl}
  \ttT_\varphi & = & 
  2\,\ttT^1 \,+\, 4\,\ttT^2 \,+\, \ttT^3 \,+\, \ttT^4 \,+\, 2\,\ttT^5 \,, \quad\hbox{or}\\
  \ttT_\varphi & = & 
  2\,\ttT^1 \,+\, \ttT^2 \,+\, \ttT^3 \,+\, 4\,\ttT^4 \,+\, 2\,\ttT^5 \,;
  \end{array}
\end{equation}
or \emph{\bf (b)} $V_\bC = \tw^3\bC^6$ and 
\begin{equation} \label{E:A5_w3}
  \begin{array}{rcl}
  \ttT_\varphi & = &  3\,\ttT^1 \,+\, 2\,\ttT^2 \,+\, \ttT^3 \,+\, 
  \ttT^4 \,+\, 7\,\ttT^5 \,,\quad\hbox{or}\\
  \ttT_\varphi & = & 7\,\ttT^1 \,+\, \ttT^2 \,+\, \ttT^3 \,+\, 
  2\,\ttT^4 \,+\, 3\,\ttT^5 \,.
  \end{array}
\end{equation}
\end{subequations}
\end{proposition}

\begin{proof}
Let $U$ be the irreducible $G_\bC$--module of highest weight $\m$ associated to $V_\bR$ as in \S\ref{S:RCQ}.  By \eqref{E:Umf}, $U$ is necessarily weight multiplicity-free.  By Theorem \ref{T:mf}(a), the highest weight of $U$ is either $\m = p \w_1, p\w_r$, or $\m=\w_k$ with $k=2,3,4$.  The cases $\m = p\w_1$ and $\m = p\w_r$ are symmetric, as are the cases $\m = \w_2$ and $\m = \w_4$.  So it suffices to consider the three cases $\m = p\w_1, \w_2 , \w_3$.

\medskip  

\noindent{\small{\bf (I)}}  Let's begin with the case that 
$$
  \m \ = \ p\,\w_1 \ = \ p\, \tfrac{1}{6}
  \left( 5 \s_1 + 4 \s_2 + 3 \s_3 + 2 \s_4 + \s_5 \right) \,.
$$
Then $U = \tSym^p\bC^6\not\simeq U^*$ and $V_\bC = U \op U^*$.  By Lemma \ref{L:Ar_pw1}, $2p$ divides $6!$; equivalently, $p$ divides $5\cdot3^2\cdot2^3$.  The lemma also asserts that $6$ does not divide $2p$.  Therefore,
$$
  p \ = \ 5^a\cdot2^b \,,\quad\hbox{with}\quad \ 
  0 \le a \le 1 \ \hbox{ and } \ 0 \le b \le 3 \,.
$$
By \eqref{E:m_dim} and \eqref{E:mmu}, 
\begin{equation} \label{E:5dim}
  2\,\tdim_\bC U \,-\, 1 \ = \ 
  \third \, 5^a 2^b \, 
  \left( 5 n_1 + 4 n_2 + 3 n_3 + 2 n_4 + n_5 \right)
\end{equation}
Since the left-hand side of \eqref{E:5dim} is odd, it is necessarily the case that $b=0$.  Therefore, $p = 1,5$.

Suppose that $p=5$.  Then $\tdim_\bC U = \binom{10}{5} = 6^2\cdot 7$.  Since $2\,\tdim_\bC U \,-\, 1 = 2\cdot 6^2\cdot 7 - 1 = 503$ is not divisible by $5$, \eqref{E:5dim} forces $b=0$.  Thus, $p=1$, and we have $V_\bC = \bC^6 \op (\bC^6)^*$.  With the assistance of mathematical software (such as the representation theory package LiE) one may confirm that \eqref{E:ev} holds if and only if $\ttT_\varphi$ is the first grading element of \eqref{E:A5_w1}.

A nearly identical argument with $\m = p \w_5$ in place of $\m = p \w_1$ yields $p=1$ and $\ttT_\varphi$ is the second grading element of \eqref{E:A5_w1}.

\medskip  

\noindent{\small{\bf (II)}}  Next, consider the case that 
$$
  \m \ = \ \w_2 \ = \ \third( 2\s_1 + 4\s_2 + 3\s_3 + 2\s_4 + \s_5) \,.
$$
We have $U = \tw^2\bC^6\not\simeq \tw^4(\bC^6) = U^*$ and $V_\bC = U \op U^*$.  From \eqref{E:m_dim} and $\tdim_\bC U = \binom{6}{2} = 15$, we obtain
$$
  m \ = \ \half 29 \,.
$$
Then \eqref{E:mmu} implies 
$$
  \half 29 \ = \ \third( 2n_1 + 4n_2 + 3n_3 + 2n_4 + n_5) \,.
$$
This is not possible, as $2$ does not divide $29\cdot3$.

\medskip

\noindent{\small{\bf (III)}}  Finally, we consider the case that 
$$
  \m \ = \ \w_3 \ = \ \half ( \s_1 + 2\s_2 + 3\s_3 + 2\s_4 + \s_5) \,.
$$
We have $U = \tw^3\bC^6 \simeq U^*$.  By \S\ref{S:RCQ}, $U$ is either real or quaternionic.  By \eqref{E:RC}, $U$ is necessarily real, so that $V_\bC = U$.  Again, with the assistance of LiE, one may confirm that \eqref{E:ev} holds if and only if \eqref{E:A5_w3} holds.  

For the first grading element of \eqref{E:A5_w3}, we have $\ttT^\tcpt = 2\,\ttT^2$, cf. \eqref{E:Tcpt}.  Therefore, $\m(\ttT^\tcpt) = 2$, and $U$ is a real representation of $G_\bR$ by \eqref{E:RQtest}, and as required by \eqref{E:RC}.  Likewise, for the second grading element of \eqref{E:A5_w3}, we have $\ttT^\tcpt = 2\,\ttT^4$ and $\m(\ttT^\tcpt) = 2$; again, $U$ is a real representation.  
\end{proof}

\subsection{The standard representation: examples} \label{S:Aeg}

\begin{example*}[Rank seven]
Let $G$ be a Hodge group with complex Lie algebra $\fg_\bC = \fsl_8\bC$.  Assume that $(V,\varphi)$ is a Hodge representation with the property that $V_\bR$ is an irreducible $G_\bR$--module, and $V_\bC = \bC^8 \op (\bC^8)^*$.  Let $\ttT_\varphi$ be the associated grading element (\S\ref{S:GE}), and let $\m = \w_1$ be the highest weight of $\bC^8$.  Assume the normalizations \eqref{E:n>=0} and \eqref{E:m*<m} hold.  A computation with LiE indicates that the Hodge representation is principal if and only if $\ttT_\varphi = n_i \ttT^i$ is given by one of
\begin{eqnarray*}
  (n_1,n_2,\ldots,n_8) & = & (1,5,1,3,1,1,1) \,,\ 
  (2,3,2,2,2,1,2) \,,\\
  &  & (3,1,3,2,1,3,1) \,,\ (3,2,1,2,3,2,1) \,.
\end{eqnarray*}

Suppose, on the other hand, that $\m = \w_7$ and the normalizations \eqref{E:n>=0} and \eqref{E:m*<m} hold.  Then $\ttT_\varphi = \sum_{i} n_{8-i}\,\ttT^i = n_7 \ttT^1 + n_6 \ttT^2 + \cdots + n_2\ttT^6 + n_1\ttT^7$, where $(n_1,\ldots,n_7)$ is one of the four above.
\end{example*}

\begin{example*}[Rank nine]
Let $G$ be a Hodge group with complex Lie algebra $\fg_\bC = \fsl_{10}\bC$.  Assume that $(V,\varphi)$ is a Hodge representation with the property that $V_\bR$ is an irreducible $G_\bR$--module, and $V_\bC = \bC^{10} \op (\bC^{10})^*$.  Let $\ttT_\varphi$ be the associated grading element (\S\ref{S:GE}), and let $\m = \w_1$ be the highest weight of $\bC^{10}$.  Assume the normalizations \eqref{E:n>=0} and \eqref{E:m*<m} hold.  A computation with LiE indicates that the Hodge representation is principal if and only if $\ttT_\varphi = n_i \ttT^i$ is given by one of
\begin{eqnarray*}
  (n_1,n_2,\ldots,n_{10}) & = & 
   (1,2,6,2,1,1,1,1,2) \,,\ (1,3,4,2,2,1,1,2,1) \,,\\
   & & (1,4,2,3,1,2,2,1,1) \,,\ (2,1,5,2,2,1,1,1,3) \,,\\
   & & (2,2,3,3,1,2,1,2,2) \,,\ (2,3,1,4,1,1,3,1,2) \,,\\
   & & (2,4,1,1,1,5,1,1,2) \,,\ (3,1,2,4,1,1,2,3,1) \,,\\
   & & (3,2,2,1,1,4,2,2,1) \,,\ (4,1,1,2,1,3,4,1,1) \,.
\end{eqnarray*}
\end{example*}

Suppose, on the other hand, that $\m = \w_9$ and the normalizations \eqref{E:n>=0} and \eqref{E:m*<m} hold.  Then $\ttT_\varphi = \sum_{i} n_{10-i}\,\ttT^i = n_9 \ttT^1 + n_8 \ttT^2 + \cdots + n_2\ttT^8 + n_1\ttT^9$, where $(n_1,\ldots,n_9)$ is one of the ten above.

\bibliography{refs.bib}
\bibliographystyle{plain}
\end{document}